\documentclass[a4paper,12pt]{article}
\usepackage[T1]{fontenc}
\usepackage[utf8]{inputenc}
\usepackage{hhline}
\usepackage[round]{natbib}
\usepackage{amssymb}
\usepackage{enumitem}
\usepackage{mathtools}
\usepackage{amsfonts}
\usepackage{amsmath}
\usepackage{amsthm}
\usepackage{thmtools}
\usepackage{mathrsfs}
\usepackage{esint}
\usepackage{latexsym}
\usepackage{multirow}
\usepackage{bm,bbm}
\usepackage{float}
\usepackage{multicol}
\usepackage[pdftex]{graphicx}
\usepackage{appendix}
\usepackage{lmodern}
\usepackage{xr}
\usepackage{microtype,lastpage}
\usepackage{subcaption}
\usepackage[left=2.5cm,right=2.5cm]{geometry}
\usepackage{xcolor}
\usepackage{hyperref}
\hypersetup{
colorlinks=true,
linkcolor = blue,
urlcolor  = blue,
citecolor = blue,
anchorcolor = blue}

\allowdisplaybreaks

\numberwithin{equation}{section}

\theoremstyle{plain}
\newtheorem{theorem}{Theorem}

\newtheorem{corollary}[theorem]{Corollary}

\newtheorem{lemma}[theorem]{Lemma}

\newtheorem{proposition}[theorem]{Proposition}
\newtheorem{remark}[theorem]{Remark}

\DeclareSymbolFont{bbold}{U}{bbold}{m}{n}
\DeclareSymbolFontAlphabet{\mathbbold}{bbold}

\definecolor{bittersweet}{rgb}{1.0, 0.44, 0.37}
\definecolor{electricultramarine}{rgb}{0.25, 0.0, 1.0}
\definecolor{rred}{RGB}{152,0,0}
\definecolor{carminepink}{rgb}{0.92, 0.3, 0.26}

\renewcommand\labelenumi{(\roman{enumi})}
\renewcommand\theenumi\labelenumi

\newcommand*\diff{\mathop{}\!\mathrm{d}}
\newcommand{\Ex}{\mathbb{E}}

\newcommand{\ind}[1]{\mathbbold{1}{\bbra{#1}}}

\providecommand{\cc}[1]{\overline{#1}}
\providecommand{\abs}[1]{\left\vert#1\right\vert}			
\providecommand{\norm}[1]{\left\Vert#1\right\Vert}			
\newcommand{\Var}{{\operatorname{Var}}}
\newcommand{\Td}{\mathbb{T}^d}

\newcommand{\fm}{f^{\left(m\right)}}
\newcommand{\N}{n}
\newcommand{\reals}{\mathbb{R}}					          
\newcommand{\integers}{\mathbb{Z}}
\newcommand{\naturals}{\mathbb{N}}
\newcommand{\complex}{\mathbb{C}}

\newcommand{\Ltwo}{{L^2\bra{\Td}}}
\newcommand{\Lp}{{L^p\bra{\Td}}}
\newcommand{\Lr}{{L^r\bra{\Td}}}
 
\newcommand{\Linf}{{L^\infty\bra{\Td}}}
\newcommand{\besov}{\mathcal{B}_{r,q}^s}
\newcommand{\besovm}{\mathcal{B}_{r,q}^{s+\mabs}}
\newcommand{\besovgen}[3]{\mathcal{B}_{{#1},{#2}}^{#3}}
\newcommand{\Tcoord}{\left(\vartheta_1,\ldots,\vartheta_d \right)}

\newcommand{\angles}[1]{\langle #1 \rangle}

\renewcommand{\theta}{\vartheta}
\newcommand{\Sn}{S_{\ell}}    
\newcommand{\Snprime}{S_{\ell^{\prime}}}    
\newcommand{\Snc}{\cc{\Sn}}
\newcommand{\an}{a_{\ell}} 
 
\newcommand{\nindex}{\left(\ell_1,\ldots,\ell_d\right)}
\newcommand{\mabs}{\abs{m}}
\newcommand{\cubew}{\lambda_{j,k}}
\newcommand{\cubep}{\xi_{j,k}}
\newcommand{\Lambdaj}{\Lambda_j^d}		
\providecommand{\needlet}[1]{\psi_{j,k}\left(#1\right)}	
\providecommand{\psijk}{\psi_{j,k}}
\providecommand{\bfun}[1]{b\left(#1\right)}	
\providecommand{\bfunm}[1]{b_j^{\left(m\right)}\left(#1\right)}
\newcommand{\betac}{\beta_{j,k}}        
\newcommand{\betacm}{\beta_{j,k}^{\left(m\right)}}

\newcommand{\nnorm}{\varepsilon_\ell}

\newcommand{\barg}{\frac{\nnorm}{B^{j}}}

\newcommand{\Kj}{K_j}

\newcommand{\psijkm}{\psi_{j,k}^{\left(m\right)}}

\newcommand{\Nzero}{\mathbb{N}_{0}}
\newcommand{\sumn}{\sum_{\ell\in \Lambdaj}}
\newcommand{\sumntot}{\sum_{\ell \in \integers^d}}
     
\newcommand{\sumjk}{\sum_{j\in\Nzero} \sum_{k=1}^{K_{j}}}
\newcommand{\sumj}{\sum_{j\in\Nzero}}		
\newcommand{\sumjtrunc}{\sum_{j= 0}^{\Jn-1}}	  
\newcommand{\sumk}{\sum_{k=1}^{K_j}}

\newcommand{\Js}{J_{s,m}}

\newcommand{\npr}{\ell^{\prime}}

\newcommand{\Dm}{D^{m}}

\newcommand{\betamest}{\widehat{\beta}_{j,k}^{\left(m\right)}}
\newcommand{\fmest}{\widehat{f}^{\left(m\right)}}
\newcommand{\fest}{\widehat{f}^{\left(m\right)}}
\newcommand{\Jn}{J_{\N,m}}
\newcommand{\thres}{\tau_{j,m,\N}}

\providecommand{\bra}[1]{\left(#1\right)}			      
			    
\providecommand{\bbra}[1]{\{#1\}}			      
\renewcommand{\(}{\left(}
\renewcommand{\)}{\right)}
\renewcommand{\[}{\left[}
\renewcommand{\]}{\right]}
\renewcommand{\{}{\left\lbrace}
\renewcommand{\}}{\right\rbrace}

\def\sto{\rightarrow}

\def\1{\mathbbm{1}}
\renewcommand{\epsilon}{\varepsilon}

\begin{document}
	\title{Nonparametric needlet estimation for partial derivatives of a probability density function on the $d$-torus} 
	\author{Claudio Durastanti \\
		\textit{La Sapienza Universit\`a di Roma}\\
		{\small claudio.durastanti@uniroma1.it} \and Nicola Turchi \\
		\textit{Universit\`a degli Studi di Milano-Bicocca}\\
		{\small nicola.turchi@unimib.it}}
	\date{\today}
	\maketitle
	\begin{abstract}
This paper is concerned with the estimation of the partial derivatives of a probability density function of directional data
on the $d$-dimensional torus within the local thresholding framework. The estimators here introduced are built by means of the toroidal needlets, a class of wavelets characterized by excellent concentration properties in both the real and the harmonic domains. In particular, we discuss the convergence rates of the $L^p$-risks for these estimators, investigating their minimax properties and proving their optimality over a scale of Besov spaces, here taken as nonparametric regularity
function spaces.
	\end{abstract}
\begin{itemize}
	\item \textbf{Keywords and Phrases: }Local thresholding, needlets, directional data, nonparametric density estimation, Besov spaces, adaptivity.
	
	\item \textbf{AMS Classification: 62G08, 62G20, 65T60}
\end{itemize}	

	\maketitle
	\section{Introduction and background}\label{sec:intro}

	
	The estimation of derivatives of the probability density function is related to several open problems in statistics. Estimators of the first order derivatives in the unidimensional framework are exploited to detect the modes of uni-modal distributions (see, among others, \cite{parzen,schuster}). The straightforward generalization to the multivariate case has led to the mean--shift algorithm (see, for example, \cite{fukunaga,silverman}), where the estimation of the gradient vector of the density function is exploited to cluster and filter data. This algorithm has become widely popular in several research fields, such as image analysis and segmentation (see, among others, \cite{cheng,comanciu}). In the same setting, estimators of second order derivatives of a probability density function are used to perform statistical tests for modes of the data density and to identify key characteristics of the distribution, such as local and global extrema, ridges or saddle points (see, for example, \cite{GPVW}).  These estimators are also extensively used in other statistical problems, such as establishing the optimal bandwidth in the framework of kernel density estimation, Fisher information estimation, parameter estimation, regression problems, hypothesis testing and others (see, among others, \cite{singh}).\\
	In the nonparametric setting, efforts have been made to exploit kernel estimation for the derivatives of a density function, even if the excellent properties of kernel density estimators are partially lost due to the problem of bandwidth/smoothing parameter selection (see, among others, \cite{chacon1}).\\ Several estimators for the derivatives of density functions on $\reals$ or $\reals^d$ have been built by means of wavelet systems and have already been proposed as alternatives to the kernel methods. This approach have been initially exploited to deal with the estimation of unknown density functions and regression functions (see, for example, \cite{WASA,tsybakov}), to be then extended to the estimation of higher order derivatives of probability density functions. Among others, wavelet estimators have been defined first on $\reals$ in \cite{prakasa0} (see also \cite{prakasaISE}), and then generalized to $\reals^d$ in \cite{prakasa1} (see also \cite{dimarzio}). Several linear and nonlinear estimators for derivatives of probability density functions on $\reals$ and $\reals^d$ have been studied, among others, by \cite{HDN11,HDN12,LiuWang,prakasa17,Prakasa18,Xu}.\\
	The so-called needlets, a second generation wavelet system, have already been extensively used in nonparametric statistics, in view of their extraordinary concentration properties in both the real and the frequency domains. Needlets have been initially built over the $d$-dimensional sphere in \cite{npw1,npw2}, while some of their stochastic properties, mostly related to spherical random fields, are discussed in \cite{bkmpAoS,dlmejs,dmp,bdmp,cammar,st22}. Needlets have also been established over general compact manifolds in \cite{gm2,knp,pesenson}, and over spin fiber bundles (see \cite{gelmar}), while a generalized construction of spherical needlets admitting additional flexibility in the harmonic domain has been introduced more recently in \cite{dmt21}.\\ The pioneering results concerning needlet frames in the setting of nonparametric statistics, described in \cite{bkmpAoSb}, have established minimax rates of convergence for the $L^p$-risk of needlet density estimators built by means of hard local thresholding techniques. Analogous results in the block and global thresholding framework were presented in \cite{durastanti2,durastanti6}. Nonparametric regression estimators of spin functions have been discussed in \cite{dmg}. The reader is referred to \cite{bd2, gautier,knp} for other relevant applications in the nonparametric paradigm. As far as the $d$-dimensional torus $\Td$ is concerned, toroidal needlets have already been discussed and applied in the framework of the two sample problem, in \cite{bd1}.\\ 
	The goal of this paper is to study the asymptotic optimality of adaptive nonlinear wavelet estimators for the derivatives of a probability density function defined over the $d$-dimensional torus $\Td$ in the nonparametric setting. We combine local thresholding techniques and concentration properties of the needlets to construct estimators achieving optimality of the $L_p$-risk for probability density functions defined in some scales of Besov spaces. The rates here obtained are consistent with the results given in \cite{bkmpAoSb,dmg}, where local thresholding techniques were applied to estimate density of spherical data, as well as those related to the estimation of derivatives of densities on $\reals$ and $\reals^d$ (see, for instance, \cite{prakasa1,LiuWang}).	
	As already remarked in \cite{bd2}, the choice of the $d$--dimensional torus for the support of the probability density function is quite general: since $\Td$ is locally homeomorphic to $\reals^d$, hence to any other manifold with the same dimension, the spatial localization property of the toroidal needlets ensures that the results here achieved remain valid when approximating locally any such manifold with $\Td$. Furthermore, as a consequence of the strong concentration properties of the needlet frames, the estimators here proposed is asymptotically robust even if defined only on sub-regions of the torus (see for further details \cite{bkmpBer,MaPeCUP}).\\
	The importance of studying and modeling multivariate toroidal data naturally arises in many practical problems. Indeed, it is natural to describe the time evolution of $(d-1)$-toroidal variables as an additional dimension in a $d$-dimensional setting (see \cite{bott} for $d=2$). This fact leads to some applications of interest, for example, in meteorology, where wind directions can be modeled as toroidal data (see, among others, \cite{gp13}). Most actual wind speeds are shown to lie near the peaks of the predicted density functions (see, for example \cite{energies}). Then the study of the gradient of the corresponding density function by means of estimators with good concentration properties can provide an excellent method to identify those peaks. Some other applications can be found in astrophysics. Studying locally 3--dimensional density functions of galaxies and their geometric characteristics is crucial to understand morphology--density, color--density and color--concentration--density relations  (see for example \cite{ferdosi} and references therein). 
	 A fascinating classification problem, introduced by \cite{biology} and mentioned by \cite{dimarzio}, concerns measurements on the skull of members from two different population groups represented by a front angle and a side angle. Identifying different locations for the peaks can be helpful to support their classification results. 
	Another interesting example comes from econometrics. In \cite{economy}, methods from functional
	principal component analysis are used to estimate partial derivatives of multivariate curves and then
	applied to the study of state price density surfaces.\\
	The structure of this paper is as follows. Section \ref{sec:prelim} introduces some preliminary
	results on the harmonic analysis on the torus,  toroidal needlet frames, Besov
	spaces and related properties. In Section \ref{sec:main}, we  present the construction of our estimators and the statements of our main results. 
	Section \ref{sec:num} contains some numerical evidence, while all the proofs are collected in Section \ref{sec:proof}.

\section{Some preliminary results}\label{sec:prelim}
In this section, we will provide the reader with some background concerning harmonic analysis on the torus, toroidal needlets, their derivatives, and Besov spaces.
We begin by setting some necessary notation. Throughout this paper, given two real-valued sequences $\{x_k\}_{k \in \naturals}$ and $\{y_k\}_{k \in \naturals}$, we write $x_k \lesssim y_k$ or $x_k \gtrsim y_k$ if there exists an absolute positive constant $c \in \reals$ such that, for any $k \in \naturals$, $x_k \leq cy_k$ or $x_k \geq cy_k$ respectively. The notation $x_k \approx y_k$ indicates that both $x_k\lesssim y_k$ and $x_k \gtrsim y_k$ hold. Furthermore, for any $z \in \mathbb{C}$, $\bar{z}$ denotes its conjugate.

\subsection{Harmonic analysis on the torus}
As is well known in the literature (see, for example, \cite{grafokos}), the $d$-dimensional torus $\Td$ can be read as the direct product of $d$ unit circles, 
$$
\Td = \mathbb{S}^1 \times \ldots \times \mathbb{S}^1 \subset \complex^d.
$$ 
From now on, we will denote the generic coordinates over $\Td$ by $\theta=\Tcoord$, where $\theta_i\in\left[\left.0,2\pi \right)\right.$ for $i=1,\ldots,d$. 
As a straightforward consequence, the uniform Lebesgue measure over $\Td$ will be given by 
$$
\rho\left(\diff \theta\right)=\prod_{i=1}^d \diff \theta _i,
$$ 
where for $i = 1,\ldots, d$, $\diff \theta _i$ is the Lebesgue measure over the unit circle. \\
Let $\langle \cdot, \cdot \rangle$ denote the standard scalar product between
$d$-dimensional vectors; the set of functions $\Sn : \Td \sto \complex,$ where $\ell = \nindex \in \integers^d$ is the frequency index, defined by
$$
\Sn \left(\theta\right) = \left(2\pi\right)^{-\frac{d}{2}} \exp \left(\angles {\ell,\theta}\right),
$$
describes an orthonormal basis for $\Ltwo$, the space of square-integrable functions over $\Td$ (see again \cite{grafokos}). Indeed, the one-dimensional torus $\mathbb{T}^1$ can be identified as an equivalence class of the quotient space $\reals\slash \integers$, so that the canonical representation in $\left[\left.0, 2\pi \right)\right.^d$ describes a coordinate system on $\Td$. The set $\{\Sn : \ell \in \integers^d \}$ corresponds to the eigenfunctions of the Laplace-Beltrami operator on the torus $\nabla_{\Td}$, defined by 
$$ 
\nabla_{\Td}=\sum_{i=1}^{d}\frac{\partial^2}{\partial \theta_i^2},
$$
so that $$\left(\nabla_{\Td} + \nnorm^2\right)\Sn\left(\theta\right)=0,$$ where $\nnorm=\sqrt{\sum_{i=1}^d\abs{\ell_i}^2}.$
Thus, the following orthonormality property holds
$$ 
\angles{ \Sn, \Snprime}_{\Td} = \int_{\Td}\Sn\left(\theta\right) \cc{\Snprime}  \left(\theta\right) \rho\left(\diff\theta\right) =
\delta_\ell^{\ell^\prime},
$$
where $\delta_\cdot^{\cdot}$ is the multivariate Kronecker delta. Any function $f \in \Ltwo$ can be then represented by its harmonic expansion 
$$
f\left(\theta\right) = \sumntot \an\Sn\left(\theta\right),  \quad \theta \in \Td,
$$ where 
$\{\an : \ell \in \integers^d \}$ is the set of the complex-valued Fourier coefficients, given by
$$
\an = \int_{\Td} f \(\theta\) \Snc \(\theta\) \rho \(\diff \theta\).
$$
Given the multi-index $m=\left(m_1,\ldots,m_d\right)\in \naturals_0^d$ such that $\mabs=\sum_{i=1}^{d}m_i$, for $f\in C^m\left(\Td\right)$, the $m$-th order derivative of $f$ is defined by
$$
\fm\(\theta\)= \Dm f\(\theta\) = \frac{\partial^{\mabs}}{\partial\theta_1^{m_1}\ldots\partial\theta_d^{m_d}} f\(\theta\),
$$ 
where the differential operator $\Dm$ is given by $$\Dm= \frac{\partial^{\mabs}}{\partial\theta_1^{m_1}\ldots\partial\theta_d^{m_d}}.$$
\subsection{Toroidal needlets}
As already mentioned in Section \ref{sec:intro}, needlets have been originally introduced on the $d$-dimensional sphere in \cite{npw1,npw2}, and then generalized to compact manifolds (see \cite{bd2,gm2,knp}). Needlet-like wavelets on $\Td$ have been already used in \cite{bd1} in the framework of the two--sample problem. Their construction can be resumed as follows.\\ 
		Fixed a resolution level $j \in \naturals$, by means of the Littlewood-Paley decomposition on $\Td$ (see \cite{npw2}), there exists a set of cubature points and weights 
		$$
			\{ \(\cubep,\cubew\): k = 1, . . . , \Kj\},
		$$
		where $\cubep \in \Td$, $\cubew \in \reals^+$, and $\Kj$ is the cardinality of needlets at the level $j$. Loosely speaking, $\Td$ can be decomposed into a partition of $\Kj$ sub-regions, called pixels, centered on the corresponding $\cubep$ and with area equal to $\cubew$. The $d$-dimensional toroidal needlets are defined by
		$$
			\needlet{\theta} = \sqrt{\cubew} \sumntot \bfun{\barg} \Snc\(\cubep\)\Sn\(\theta\),
		$$ 
		where $B>1$ is a a scale parameter (typically, $B=2$), and $b : \reals \sto \reals^{+}$ is the so-called needlet weight or window function, which satisfies the following properties:
		\begin{enumerate}
			\item $b$ has compact support in $\[ B^{-1},B\]$;
			\item $b \in C^\infty \(\reals\)$;
			\item the partition of unity property holds, that is, for any $c > 1$, 
			$$
				\sum_{j\in \naturals} b^2\(\frac{c}{B^j}\)=1.
			$$
		\end{enumerate}
		Therefore, needlets are characterized by the following properties. Following (i), for any $j \in \naturals$, $b(\barg)$ is not null only over a finite subset of $\integers^d$, that is $\Lambdaj = \{\ell: \nnorm \in \(B^{j-1}, B^{j+1}\)\}$. As a direct consequence, we can rewrite
		$$
			\needlet{\theta} = \sqrt{\cubew} \sumn \bfun{\barg} \Snc\(\cubep\) \Sn\(\theta\).
		$$
		In view of (ii), toroidal needlets are characterized by a quasi-exponential localization property in the spatial domain, which guarantees that each needlet $\psijk$ is not--negligible almost only in the corresponding pixel. More rigorously, for any $M >0$, there exists $c_M > 0$ such that, for any $\theta \in \Td$,
		$$
			\abs{\needlet{\theta}} \leq \frac{c_M B^{\frac{d}{2}j}} {\(1 + B^{j} d \(\theta, \cubep\)\)^M},
		$$
		where $d \(\cdot, \cdot\)$ is the geodesic distance over $\Td$. As a consequence, the following bounds on the $L^p$-norms of the toroidal needlets hold (see \cite{npw2}): for any $p \in \[\left. 1, \infty \)\right.$, there exist two positive constants $c_p, C_p$, which depend only on $p$,  such that 
		\begin{equation}\label{eq:neednorm} 
			c_p B^{jd\(\frac{1}{2}-\frac{1}{p}\)} \leq \norm{\psijk}_{\Lp}\leq C_p B^{jd\(\frac{1}{2}-\frac{1}{p}\)}.
		\end{equation}
		Finally, it follows from (iii) that the needlet system $\{ \psijk : j \geq 0; k = 1,\ldots , \Kj\}$ is a tight frame over $\Td$. Indeed, for any function $f \in \Ltwo$, we can define the set of needlet coefficients $\{ \betac : j \geq 0; k = 1,\ldots , \Kj\}$, each of those given by
		$$ 
			\betac = \angles{f,\psijk}_{\Td} =\int_{\Td}f\(\theta\)\cc{\psijk}\left(\theta\right) \rho \(\diff \theta\).$$
		Then, it holds that 
		$$
			\sumj \sumk \abs{\betac}^2= \norm{f}_{\Ltwo}^2,
		$$
		and the following reconstruction formula holds in the $L^2$--sense
		\begin{equation*}\label{eq:recons}
			f \(\theta\) = \sumj \sumk \betac \needlet{\theta}, \quad \theta \in \Td.
		\end{equation*} 	
	\subsection{Needlets and derivatives}
		The needlet expansion of the $m$-th order derivative of the density function, $\fm=\Dm f$, is given by
	\begin{equation}\label{eq:derrec}
		\fm\(\theta\) = \sumjk \betacm \needlet{\theta}, \quad \theta \in \Td,
	\end{equation}
	where $\{\betacm:j \geq 0,k=1,\ldots,\Kj\}$ is the collection of needlet coefficients associated to $\fm$, that is,
	\begin{equation}\label{eq:dercoeff}
		\betacm=\int_{\Td} \fm \(\theta\)\cc{\psijk}\(\theta\) \rho\(\diff \theta\).
	\end{equation}
The $m$-derivative of a toroidal needlet is given by
		\begin{equation}\label{eq:multiderdef}
			\psijkm \left(\theta\right)  =  \Dm \psijk\(\theta\),
		\end{equation}
		such that the following result holds.
		\begin{lemma}\label{lemma:derivative}
			Let $\psijkm$ be given by \eqref{eq:multiderdef}. Then, it holds that
			\begin{equation}
				\psijkm\left(\theta\right) = \sqrt{\cubew} B^{j\mabs}\sumn \bfunm{\ell} \Snc\(\cubep\)\Sn\(\theta\)
				, \quad \theta  \in \Td,\label{eq:multider}
			\end{equation}
			where
			\begin{equation*}\label{eq:bm}
				\bfunm{\ell}=\left(-1\right)^{\mabs}\frac{\prod_{i=1}^{d}\ell_i^{m_{i}}}{B^{j\mabs}}b\left(\barg\right).
			\end{equation*}
		\end{lemma}
		\noindent The proof of Lemma \ref{lemma:derivative} is given in Section \ref{sec:needproof}
		\begin{remark}
Observe that the function $b_j^{\(m\)}:\reals^d \sto \reals $ preserves some of the properties of $b$. Indeed, it is $C^\infty\(\reals^d\)$ and has compact support in $\Lambdaj$. On the other hand, the partition of unity property does not hold anymore.
 		\end{remark}
		\noindent The next result is concerned with the localization property of the derivatives of a toroidal needlet.
		\begin{lemma}\label{lemma:localization}
			Let $\psijkm$ be given by \eqref{eq:multider}. Then, for any multi-index $m \in \naturals_0 ^d$ and for any $M>0$, there exists $c_{M}>0$ such that, for any $\theta\in \Td$,
			\begin{equation*}
				\abs{\psijkm \(\theta\)} \leq \frac{c_M B^{j\(\mabs+\frac{d}{2}\)}}
				{\(1 + B^{j}d \(\theta, \cubep\)\)^M}.
			\end{equation*} 
		\end{lemma}
		\noindent The following corollary establishes bounds for the $L^p$-norms of the needlet derivatives. These bounds correspond to the ones given by \eqref{eq:neednorm} for the toroidal needlets. 
		\begin{corollary}\label{cor:dernorm}
			Let $\psijkm$ be given by \eqref{eq:multider}. Then, for any multi-index $m \in \naturals_0 ^d$ and for any $p \in \left[\left. 1, \infty \right)\right.$, there exist two constants $c^\ast_p,C^\ast_p$, depending only on $p$ such that
			\begin{equation*}\label{eq:derivativenorm} 
				c^\ast_p B^{j\(\mabs+d\(\frac{1}{2}-\frac{1}{p}\)\)} \leq \norm{\psijkm}_{\Lp}\leq C^\ast_p B^{j\(\mabs+d\(\frac{1}{2}-\frac{1}{p}\)\)} .
			\end{equation*}
		\end{corollary}
		\noindent  The bound in Lemma \ref{lemma:localization} follows a general result in mathematical analysis, given by \cite[Theorem 2.2]{gm2}, which
		states that the spatial concentration properties of needlet-like constructions over compact manifolds are conserved under the action of $C^\infty$-differential operators, up to a polynomial term which depends on the degree of the operator itself, in our case $\mabs$ for $\Dm$.\\
		On the other hand, the proof of Corollary \ref{cor:dernorm} follows strictly the one for the standard needlets given by \cite{npw2}[Eqq. 3.12, 3.13] and for the Mexican needlets in \cite{durastanti1}[Corollary 3.2]. Both the proofs are then omitted for the sake of the brevity.\\
		Then, the following result holds
		\begin{lemma}\label{lemma:sumbeta} 
			Let $\psijkm$ and $\betacm$ be given by \eqref{eq:multider} and \eqref{eq:dercoeff} respectively. Then, it holds that
			\begin{equation}\label{eq:sumbeta1}
				\betacm = \(-1\)^{\mabs}\angles{f,\psijkm}_{\Td}.
			\end{equation}
			Furthermore, for any $j \geq 0$, it holds that 
			\begin{equation}\label{eq:sumbeta2}
				\sum_{k=1}^{\Kj}\betacm \psijk\(\theta\)=\sum_{k=1}^{\Kj}\betac \psijkm\(\theta\), \quad \theta \in \Td.
			\end{equation}
		\end{lemma}
	\noindent As a straightforward consequence, Equation \eqref{eq:derrec}
	becomes
	\begin{equation*}
			\fm\(\theta\) = \sumjk \betac \psijkm{\(\theta\)}, \quad \theta \in \Td,	
	\end{equation*} 	
The next Lemma collects some results which can be seen as the counterpart in our setting of  \cite[Lemma 2, Lemma 18]{bkmpAoSb}. 
	 	\begin{lemma}\label{lemma:supinf}
	 		For any $j\in \naturals$, let $\{a_k:k=1,\ldots,\Kj\}$ be a finite real-valued sequence. Hence, for any $0<p\leq \infty$, it holds
	 		\begin{equation}\label{eq:sup}
	 			\norm{\sumk a_k \psijkm}_{\Lp}\lesssim \begin{cases} B^{j\left(\mabs + d\(\frac{1}{2}-\frac{1}{p}\)\right)} \(\sumk\abs{a_k}^p\)^{\frac{1}{p}} & 0<p<\infty \\
	 			B^{j\left(\mabs + \frac{d}{2}\right)} \(\underset{k=1,\ldots,\Kj}{\sup}\abs{a_k}\) & p=\infty 
	 			\end{cases}.
	 		\end{equation} 
	 		Furthermore, there exists a subset $A_j \subset \{1,\ldots,\Kj\}$, where	 		$$ 
	 			\operatorname{card}A_j \gtrsim B^{jd},
	 		$$
	 		such that  
	 		\begin{equation}\label{eq:inf}
		 		\norm{\sum_{k \in A_j} a_k \psijkm}_{\Lp}\gtrsim \begin{cases} B^{j\left(\mabs + d\(\frac{1}{2}-\frac{1}{p}\)\right)} \(\sum_{k \in A_j}\abs{a_k}^p\)^{\frac{1}{p}} & 0<p<\infty \\
		 		B^{j\left(\mabs + \frac{d}{2}\right)} \(\underset{k \in A_j}{\sup}\abs{a_k}\) & p=\infty 
		 		\end{cases}.
	 		\end{equation} 
	 	\end{lemma} 
		\noindent As discussed below (see Formula \eqref{eq:pirupi}), Equation \eqref{eq:sup} will be crucial to establish a suitable upper bound for the needlet expansion of toroidal density functions belonging to some Besov space. As far as a lower bound is concerned, if the toroidal needlets were orthonormal, then \eqref{eq:sup} could have been immediately reversed, such that $A_j = \{1,\ldots,\Kj\}$ in \eqref{eq:inf}. Nevertheless, needlets, as well as their derivatives, are not orthogonal, but for $j\in\Nzero$ and $k,k^\prime \in \{1,\ldots,\Kj\}$ such that $\cubep$ and $\xi_{j,k^\prime}$ are distant enough, the scalar product between $\psijk$ and $\psi_{j,k^\prime}$ (and, consequently, between $\psijkm$ and $\psi_{j,k^\prime}^{\(m\)}$) is almost negligible. This reasoning can be extended pairwise to all the needlets whose cubature points belong to $A_j$. The proof of Lemma \ref{lemma:supinf} is very similar to the ones given in \cite[Lemma 2, Lemma 18]{bkmpAoSb} and, then, here omitted for the sake of brevity. 
		
		\subsection{Besov spaces}
		Before concluding this section, we recall the construction of the Besov spaces on $\Td$ and their excellent approximation properties for needlet coefficients. Further details can be found, among others, in \cite{gm2}, and in the references therein. Following \cite{gelpes} (see also \cite{bkmpAoSb,donoho1,dmg,durastanti6}, we consider $\mathcal{P}_t$ the space of polynomial functions of degree $t$ on $\Td$. For any $f:\Td \mapsto \reals$, $f \in \Lp$, the \emph{approximation error} $G_t\left(f;p\right)$, obtained when we replace $f$ by $g \in \mathcal{P}_t$, is defined by
		$$
			G_t\left(f;p\right) = \inf_{g \in \mathcal{P}_t} \norm{f-g}_{\Lp}.
		$$   
 		The Besov space $\besov$ is the space of functions such that
 		$$
 			f \in L^p\left(\Td\right) \quad \text{and} \quad \left(\sum_{t=0}^{\infty}\left(t^s G_t\left(f;r\right)\right)^q\right)^{\frac{1}{q}}<\infty.
 		$$
		Since $t \mapsto  G_t\left(f, r\right)$ is decreasing, a standard condensation argument yields the following equivalent conditions:
		$$
			f \in L^r\left(\Td\right) \quad \text{and} \quad \left(\sum_{j=0}^{\infty}\left(B^{js}G_{B^j}\left(f;r\right)\right)^q\right)^{\frac{1}{q}}<\infty.
		$$
 		The parameters of the Besov space $\besov$ can be read as follows (see also \cite{bd2}):
		 \begin{itemize}
		 	\item $r$ concerns the summability with respect to $j$. In particular, for any j > 0, $\{\beta_{j,k}: k=1,\ldots,K_j \}$ belongs to the
		 	set of $r$-summable sequences $\ell_r(\Td)$.
		 	\item $q$ is the controlling parameter related to the weighted $q$-norm along the scale of the needlet coefficients at the resolution level $j$;
		 	\item $s$ controls the smoothness of the decay rate for the weighted $q$-norm of the needlet coefficients.
		 \end{itemize}
	 Defining $s= s^\prime+ a$, where $s^\prime \in \mathbb{N}$ and $a \in \left(0,1\right)$, and following \cite[Definition 1]{Kerkypicard92} (see also \cite{gelpes}) yields
	 $$
	 f\in \besov \iff \fm \in \besovgen{r}{q}{a}, \text{ for }\mabs \leq s. 
	 $$
 		As a straightforward consequence, we have that 
		 \begin{align}\label{eq:besovder}
 			f\in \besovgen{r}{q}{s+\mabs} \iff \fm \in \besovgen{r}{q}{s}. 
 		\end{align} 
 		Then, we can prove the next result, analogous to \cite[Theorem 4]{bkmpAoSb} but properly adapted to derivative functions defined on $\Td$.
		\begin{proposition}\label{prop:pippo}
 			Let $1 \leq r \leq \infty$, $s > 0$, $0 \leq q \leq \infty$. Let $f$ be a measurable function on $\Td$, associated to the needlet coefficients  $\{\beta_{j,k}: j\in\Nzero, k=1,\ldots,K_j \}$. Then, the following conditions are equivalent 
			 \begin{enumerate}
 				\item $f \in \besovm$; 
 				\item $\fm \in \besov$;
 				\item for every $j \geq 1$,
				$$
					\sum_{k=1}^{K_j}\beta_{j,k}^r \norm{\psi_{j,k}}_{\Lr}^r=B^{-j\left(s+\mabs\right)}\delta_j,
				$$
				where $\left(\delta_j: j\in\Nzero \right)$ is a $q$-summable sequence;
 				\item for every $j \geq 1$,
				$$
					\sum_{k=1}^{K_j}\left(\beta^{\left(m\right)}_{j,k}\right)^r \norm{\psi_{j,k}}_{\Lr}^r=B^{-js}\delta_j,
				$$
				where $\left(\delta_j: j\in\Nzero \right)$ is a $q$-summable sequence.
 				\item for every $j \geq 1$,
				$$
					\sum_{k=1}^{K_j}\beta_{j,k}^r \norm{\psi^{\left(m\right)}_{j,k}}_{\Lr}^r=B^{-js}\delta_j.
				$$
				where $\left(\delta_j: j\in\Nzero \right)$ is a $q$-summable sequence.
			\end{enumerate}
		\end{proposition}
		\noindent The proof of this Proposition follows directly \eqref{eq:besovder}, \eqref{eq:sumbeta2} in Lemma \ref{lemma:sumbeta}, as well as \eqref{eq:sup} and the proof of \cite[Theorem 4]{bkmpAoSb}, so it is here omitted for the sake of brevity.\\
		Using jointly \eqref{eq:sumbeta2} in Lemma \ref{lemma:sumbeta}, \eqref{eq:sup} in Lemma \ref{lemma:supinf} and Proposition \ref{prop:pippo} leads to the following inequality, for any $j \in \naturals$, 
		\begin{align}\label{eq:pirupi}
			\norm{\sumk \betacm \psijk}_{\Lp} \lesssim B^{j\(\mabs+d(\frac{1}{2}-\frac{1}{p})\)} \(\sumk\abs{\betac}^p\)^{\frac{1}{p}}. 
		\end{align}
		Following \cite{bkmpAoSb}, the Besov space $\besov$ can be read as a Banach space with norm 
		\begin{equation*}
			\norm{f}_{\besov} =\norm{\left(B^{j\left[s+d\left(\frac{1}{2}-\frac{1}{r}\right)\right]} \norm{\left(\betac\right)_{k=1,\ldots,\Kj}}_{\ell^r}\right)_{j\in\Nzero}}_{\ell^q}<\infty,
		\end{equation*}
		so that we can define the \emph{Besov ball} of radius $L>0$ as the following set:
		\begin{equation*}\label{eq:besovball}
			\besov\left(L\right)=\{f \in \besov: \norm{f}_{\besov}\leq L\}.
		\end{equation*}
		Finally, as stated in \cite[Theorem 5]{bkmpAoSb} (see also \cite{Kerkypicard93,Kerkypicard}), the following \emph{Besov embeddings hold} for $p<r$
		\begin{align*}
			&\besovgen{r}{q}{s}	\subset		\besovgen{p}{q}{s}\\
			&\besovgen{p}{q}{s}	\subset		\besovgen{r}{q}{s-d\left(\frac{1}{p}-\frac{1}{r}\right)},			
		\end{align*}
		which can be equivalently stated as
		\begin{align*}
			&\sumk \abs{\betac}^{r} \leq \sumk \abs{\betac}^p\\
			&\sumk \abs{\betac}^{p} \leq \left(\sumk \abs{\betac}^r\right) \Kj^{1-\frac{p}{r}}.
		\end{align*}
		A detailed proof is similar to the ones in \cite[Theorem 5]{bkmpAoSb} and \cite[Equation 8]{dmg}.

\section{Local thresholding via toroidal needlets}\label{sec:main}
	\subsection{The construction of the estimators}
	Let $X_1,\ldots,X_{\N}$ be independent and identically distributed random vectors on $\Td$ with unknown density $f:\Td \mapsto \left[\left. 0,\infty\right.\right)$, $f \in C^m\(\Td\)$, whose upper bound is given by 	
	$$
	M=\sup_{\theta \in \Td} \abs{f\(\theta\)}.$$ 
	We assume that $f$ has derivatives of order $\mu \in \naturals$. Our goal is to produce an estimator $\fmest$ for $\fm$, for $m\in \mathbb{N}^d$ such that $\mabs\leq \mu$. The first step will consist in defining empirical estimators for the needlet coefficients; the second step will be focused on the thresholding procedure which yields the construction of the target estimator.  	Analogously to \cite{prakasa1} (cf. also \cite{prakasa0,prakasaISE}), for any $j \geq 0$ and $k \in  \{1,\ldots , \Kj\}$, we can define an \emph{empirical estimator} for the $m$-th derivative needlet coefficients. 
		\begin{equation}\label{eq:betamest}
			\betamest = \frac{\(-1\)^{\mabs}}{\N} \sum_{i=1}^{\N} \cc{\psijkm}\(X_i\) .
		\end{equation}
	We exploit the empirical needlet coefficients \eqref{eq:betamest} to define nonlinear estimators for derivatives of a density function as follows. First, we apply a selection procedure to each coefficient, and then only the selected ones will be used to construct the estimator of $\fm$ given by the following formula:
		\begin{equation}\label{eq:esti}
			\fest\left(\theta\right) = \sumjtrunc\sumk \eta\left(\betamest, \thres\right) \psijk\left(\theta\right),
		\end{equation}	
		where $\Jn$ and $\thres$ are the so-called \emph{truncation bandwidth} and \emph{threshold} respectively, while $\eta:\reals\times\reals \sto \reals$ is the \emph{thresholding function}.\\ 
		The truncation bandwidth corresponds to the higher resolution level on which the empirical coefficients are computed. The optimal choice for the truncation bandwidth is defined as follows:
		\begin{equation}\label{eq:trunc}
		\Jn=\Bigl\lfloor\frac{1}{d+2\mabs}\log_B \frac{\N}{\log \N}\Bigr\rfloor.
		\end{equation}
		Loosely speaking, \eqref{eq:trunc} allows us to control, in an optimal way, the error due to the approximation of the infinite sum $\fm$ by the finite sum $\fmest$. Further details concerning choice of $\Jn$ will be provided by Remark \ref{rem:optimal} below. \\
		The thresholding function
		is aimed to control the selection of the empirical needlet coefficients (see, for example, \cite{donoho1}). In the framework of \emph{local thresholding}, where each empirical coefficient is examined separately, it is standard to choose between \emph{hard} and \emph{soft} thresholding (see, for example, \cite[Chapters 10 and 11]{WASA}). Hard thresholding either keeps or discards the empirical coefficients. The soft thresholding, also known as \emph{wavelet shrinkage}, ``shrinks'' towards zero the values for the empirical coefficients. The two thresholding functions $\eta_{\text{hard}}:\reals\times\reals \sto \reals$ and $\eta_{\text{soft}}:\reals\times\reals \sto \reals$ are given respectively by
		\begin{align}\label{eq:thres}
			\eta_{\text{hard}}\left(u,a\right)&=\begin{cases}
			u & \text{if }\abs{u}\geq a\\
			0 & \text{otherwise}
			\end{cases},\\
		\label{eq:thressoft}
			\eta_{\text{soft}}\left(u,a\right)&=\max\left(\abs{u}-a,0\right)\operatorname{sign}\(u\),
		\end{align}  
	see Figure \ref{fig:1}.
		\begin{figure}\centering
			\includegraphics[width=0.6\textwidth]{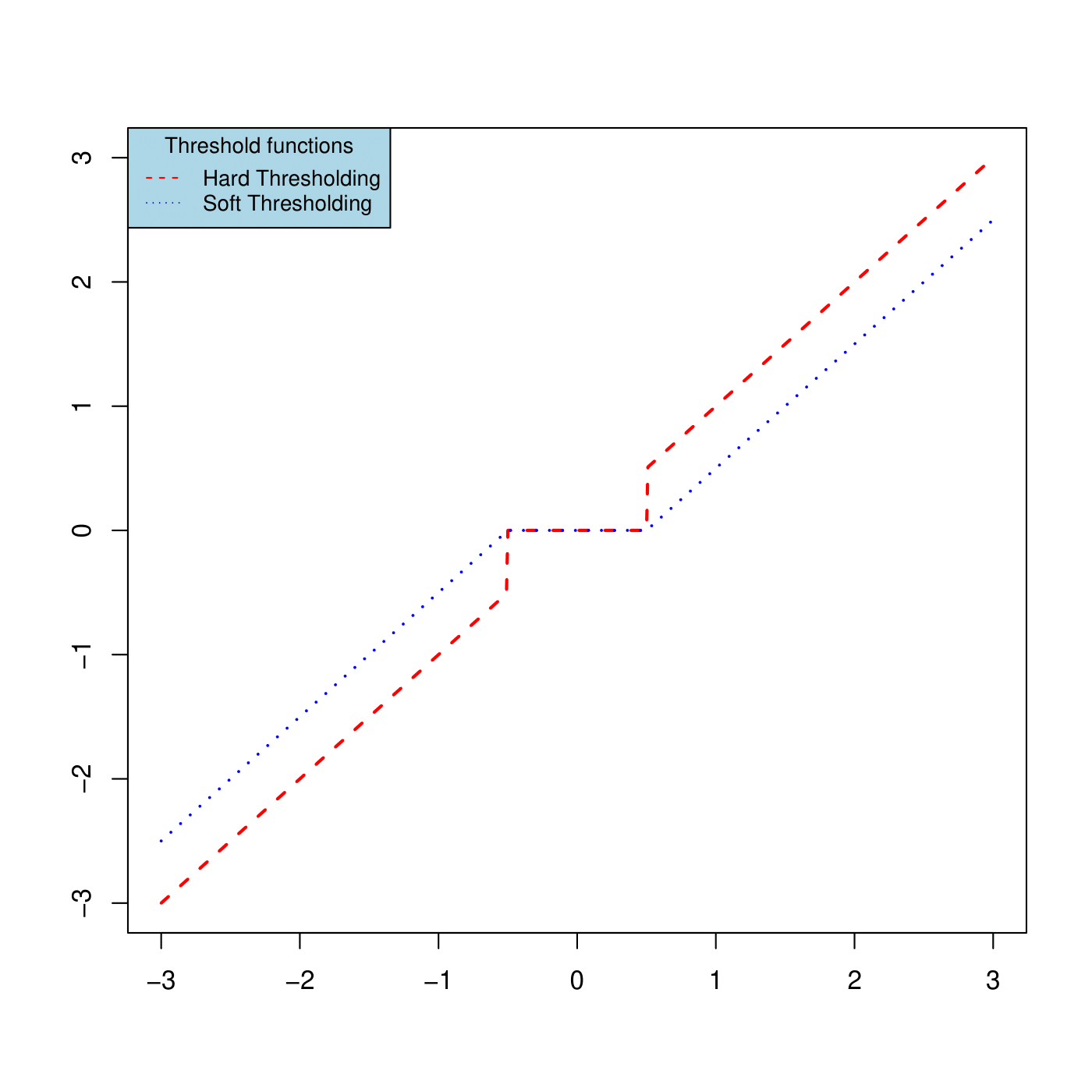}
			\caption{Comparison between hard (red line) and soft (blue line) threshold functions ($a=0.5$).}
			\label{fig:1}
		\end{figure}
		The \emph{hard} and the \emph{soft thresholding needlet estimators for the $m$-th derivative of the density $f$}, at every $\theta \in \Td$, are then respectively defined by
		\begin{equation*}\label{eq:estihard}
			\fest_{\text{hard}}\left(\theta\right) = \sumjtrunc \sumk \eta_{\text{hard}}\left(\betamest, \thres\right) \psijk\left(\theta\right),
		\end{equation*}
		and
		\begin{equation*}\label{eq:estisoft}
			\fest_{\text{soft}}\left(\theta\right) = \sumjtrunc \sumk\eta_{\text{soft}}\left(\betamest, \thres\right) \psijk\left(\theta\right),
		\end{equation*}
		where the tuning parameters are described as follows. 
			The \emph{threshold} $\thres$ is the product of three objects:
			\begin{equation}\label{eq:thresini}
				\thres=\kappa B^{j\mabs}\sqrt{\frac{\log n}{n}},
			\end{equation}
			where 
			\begin{itemize}
				\item the \emph{threshold constant} $\kappa$, which depends on the Besov parameters. Our theoretical setting will not answer the question how to choose $\kappa$ (we will show only that $\kappa$ should be large enough). However, in Section \ref{sec:num} we will provide some examples; 
				\item the \emph{derivative balancing parameter} $B^{j\mabs}$, which depends on the order of the differential operator $D^m$; 
				\item the \emph{sample size--dependent scaling factor} $\sqrt{\frac{\log n}{n}}$.
		\end{itemize}
	\begin{remark}
		Another common choice in the literature concerning the sample size--depending scaling factor is $\sqrt{j/\N}$, rather than $\sqrt{\log \N / \N}$ (see, among others, \cite{donoho1,Kerkypicard92,trib95}). As shown for example by \cite[Proof of Proposition 10.3]{WASA}, the two factors are equivalent, and, hence, another possible choice for the threshold is 
		$$
		\thres^\prime = \kappa^\prime \sqrt{j}B^{j\mabs}\N^{-\frac{1}{2}}.
		$$
	\end{remark}

\subsection{Some probabilistic results on the empirical needlet coefficients} 	First, observe that the estimator \eqref{eq:betamest} is \emph{unbiased}, 
\begin{align*}
	\Ex\[\betamest\]& =  \frac{\(-1\)^{\mabs}}{\N} \sum_{i=1}^{\N} \Ex\left[\cc{\psijkm}\(X_i\)\right]= \betacm.
\end{align*}
Using Corollary \ref{cor:dernorm} with $p=2$ yields the following upper bound for its variance:
\begin{equation}\label{eq:var}
	\Var\(\betamest\) \leq M\int_{\Td} \abs{\psijkm \(\theta\)}^2 \rho\(\diff \theta\) \leq C^{\ast}_2 M B^{2j\mabs}.
\end{equation}
Additional probabilistic bounds on the empirical coefficients are given in the following Lemma \ref{lemma:probbounds}, which can be read as the counterpart in our setting of \cite[Lemma 16]{bkmpAoSb}.
\begin{lemma}\label{lemma:probbounds}
	Let $\betamest$ be given by \eqref{eq:betamest}. Hence, for any $j \geq 0$ such that $B^{j\left(d+2\mabs\right)}\leq\sqrt{n}$ and for $k=1,\ldots,\Kj$, it holds that
	\begin{align}
		& \Pr\left(\abs{\betamest-\betacm}\geq x \right)\leq 2\exp\left(-\frac{\N x^2}{2B^{j\mabs}\left(C^{\ast}_2 M +\frac{1}{3}C^{\ast}_{\infty} \sqrt{\N}x  \right)}\right)\quad \text{ for }x>0;\label{eq:bernstein}\\
		& \Ex\left[\abs{\betamest-\betacm}^\eta\right]\lesssim n^{-\frac{\eta}{2}}B^{j\mabs\frac{\eta}{2}}\quad \text{ for }\eta\geq 1;\label{eq:momenteta}\\
		& \Ex\left[\sup_{k=1,\ldots,\Kj}\abs{\betamest-\betacm}^\eta\right]\lesssim \left(j+1\right)^{\eta} n^{-\frac{\eta}{2}} B^{j\mabs\frac{\eta}{2}}\quad \text{ for }\eta\geq 1.\label{eq:momentinf}
	\end{align}
\end{lemma}
\noindent The proof of Lemma \ref{lemma:probbounds} can be found in Section \ref{sec:needproof}.\\
	Finally, using Equation \eqref{eq:bernstein} in Lemma \ref{lemma:probbounds} leads to the following result. 
\begin{lemma}\label{lemma:sussig}
	Let $\thres$ be given by \eqref{eq:thresini}. For $\kappa \geq 6\frac{C_2^\ast M}{C_\infty ^\ast}$, under the hypotheses given in Lemma \ref{lemma:probbounds}, there exists $		\gamma \leq \left(\frac{3}{4C_\infty^\ast}\right)\kappa, 
	$ such that
	\begin{equation}\label{eq:sussig}
		\Pr\left( \abs{\betamest - \betacm}\geq \frac{\thres}{2} \right) \lesssim \N^{-\gamma}. 
	\end{equation}
\end{lemma}
\subsection{The main results}
		We will show that $\fest$ achieves the optimal rates of convergence up to some logarithmic factors with respect to $\Lp$-loss functions in both the hard and soft thresholding frameworks. In other words, our goal is to evaluate the global error measure for the estimator $\fmest$, by studying the worst possible performance of the $L^p$-risk $\Ex \[\norm{\fmest-\fm}_{\Lp}\]$ over a given nonparametric regularity class $\{\besovm:1<r<\infty,1\leq q\leq \infty,s>0 \}$ of function spaces, that is, the minimax rate of convergence
		$$
		\mathcal{R}_{p,\N}\(\besov\(R\)\)= \inf_{\fmest}\sup_{f\in \besovm\(R\)} \Ex \[\norm{\fmest-\fm}_{\Lp}\],
		$$
		where the infimum is computed over all the possible estimators, $1\leq p\leq \infty$ and $0<R<\infty$ is the radius of the Besov ball on which $f$ is defined. For $r>0$, we will show that $\fmest$ is adaptive for the $L^p$-risk and for the scale of classes of Besov balls $\{\besovm\(R\):1<r<\infty,1\leq q\leq \infty,s>0,0<R<\infty \}$, that is, for every choice of the parameters $r,s,q$ and the radius $R$, there exists a constant $c_{r,s,q,R} > 0$,
		such that
		$$
		\Ex \[\norm{\fmest-\fm}_{\Lp}\]\leq c_{r,s,q,R}\mathcal{R}_{p,\N}\(\besov\(R\)\),
		$$
		see, e.g., \cite[Definition 11.1]{WASA}. Furthermore, we will prove that $\fmest$ attains the optimal rate of convergence, that is, 
		$$
		\sup_{f\in \besovm\(R\)} \Ex \[\norm{\fmest-\fm}_{\Lp}\]   \approx \mathcal{R}_{p,\N}\(\besov\(R\)\),
		$$
		see \cite[Definition 10.1]{WASA}. Our achievements are described by two theorems. The first is concerned with the upper bound for the $\Lp$-risk. 
		\begin{theorem}[Upper bound]\label{thm:maintheorem}
			Let $f \in \besovm\left(R\right)$, where $s-\frac{d}{r}>0$, let $\fm = D^mf $, and let $\fest$ be defined by \eqref{eq:esti}, with $\eta$ given by \eqref{eq:thres}-\eqref{eq:thressoft}. Then, for any $1\leq p<\infty$, there exists a threshold constant $\kappa>0$ such that it holds 
			\begin{equation*}\label{eq:maintheorem}
				\sup_{f \in \besov\left(R\right)} \Ex \left[\norm{\fest-\fm}_{\Lp}^p\right] \lesssim  \left(\log n\right)^p \left[\frac{n}{\log n}\right]^{-\alpha\left(s,\mabs,p,r\right)},
			\end{equation*}
			where
			\begin{equation}\label{eq:rate}
				\alpha\left(s,\mabs,p,r\right)= \begin{cases}
				\frac{ps}{2\left(s+\mabs\right)+d} & \text{for }r\geq \frac{\left(2\mabs+d\right)p}{2\left(s+\mabs\right)+d}  \quad\text{(regular zone)}\\
				\frac{p\left(s+d\left(\frac{1}{p}-\frac{1}{r}  \right)\right)}{2\left[\left(s+\mabs\right)+d\left(\frac{1}{2}-\frac{1}{r} \right) \right]} & \text{for }r< \frac{\left(2\mabs+d\right)p}{2\left(s+\mabs\right)+d}  \quad\text{(sparse zone)}
				\end{cases}.
			\end{equation}
			Moreover, for $p=\infty$, 	
			\begin{equation*}\label{eq:maintheorem2}
				\sup_{f \in \besov\left(R\right)} E\left[\norm{\fest-\fm}_{L^\infty \left(\Td\right)} \right]\lesssim  \left[\frac{n}{\log n}\right]^{-\alpha\left(s,\mabs,\infty,r\right)},
			\end{equation*}
			where
			\begin{equation*}\label{eq:rate2}
				\alpha\left(s,\mabs,\infty,r\right)=\frac{s-\frac{d}{r}  }{2\left[\left(s+\mabs\right)-d\left(\frac{1}{r}-\frac{1}{2} \right) \right]}.
			\end{equation*}
		\end{theorem}
		\noindent The names ``regular'' and ``sparse'' zones, standard in the literature (see, again, \cite{WASA}), can be motivated as follows. In the regular zones, the hardest functions to be estimated are the ones characterized by a regular oscillatory behavior, that is, they are of a saw-teeth form. In the sparse zone the hardest function to estimate are those which are very regular everywhere but in small subsets of the domain, where they present strong irregularities. In this case, just few needlet coefficients $\betacm$ are not null, and that justifies the name ``sparse''. Observe that $p\leq 2$ corresponds always to the regular zone (see \cite{bkmpAoSb}). For further details and discussions, the reader is referred to \cite{donoho1, WASA}.\\
		As usual in the nonparametric setting, we can rewrite the $L^p$--risk as follows
		\begin{align*}
			\Ex \left[\norm{\fest-\fm}_{\Lp}^p\right] & = \Ex\left[\norm{ \sumjtrunc \eta\left(\betamest, \thres\right) \psijk\left(\theta\right) \psijk-\sumj \betac \psijk }_{\Lp}^p\right]\\
			& = \Ex\left[\norm{ \sumjtrunc \left[\eta\left(\betamest, \thres\right)-\betac \right]\psijk  -\sum_{j\geq \Jn} \betac \psijk } _{\Lp}^p\right]\!, 
		\end{align*}
		so that it can be bounded as follows
		\begin{align*}
			\Ex &\left[\norm{\fest-\fm}_{\Lp}^p\right] \\
			& \leq 2^{p-1}\left[\Ex\left[ \norm{ \sumjtrunc \left[\eta\left(\betamest, \thres\right)-\betac \right]\psijk  } _{\Lp}^p  \right]+\norm{ \sum_{j\geq \Jn} \betac \psijk } _{\Lp}^p   \right]\\
			&=2^{p-1}\left(\Sigma_p+D_p\right),
		\end{align*}
		with the natural extension for $p=\infty$.\\ The term $\Sigma_p$ can be read as the stochastic error to the replacement of the true needlet coefficients with the selected empirical ones, and $D_p$ is the deterministic error which arises when we select only a finite set of empirical coefficients (see also \cite{bkmpAoSb}). While the bias term $D_p$ does not affect the rate of convergence for $s > \frac{d}{p}$, the asymptotic behavior of the stochastic error $\Sigma_p$ is established thanks to the so-called \emph{optimal bandwidth selection}, that is, the resolution level $\Js$ defined by
		\begin{equation}\label{eq:optimalselection}
			\Js:B^{\Js} \approx \begin{cases}
			\left(\frac{n}{\log n}\right)^{\frac{1}{2\left(s+\mabs\right)+d}} & \text{(regular zone)}\\
			\left(\frac{n}{\log n}\right)^{\frac{1}{2\left(s+\mabs+d\left(\frac{1}{2}-\frac{1}{2}\right)\right)}} & \text{(sparse zone)}\end{cases}
		\end{equation}
		\begin{remark}\label{rem:optimal} Following \cite{bkmpAoSb}, but also \cite{efrom,kpt96}, both in the regular and in the sparse zones for any $k=1, \ldots \Kj$, $\abs{\betac}\geq \thres$ implies that $j\leq \Js$. Contrarily, the converse implication is not true. Anyway, in this case, the coefficient $\betamest$ should be discarded, since its error would be of order $B^{j\mabs}n^{-\frac{1}{2}}$, as shown in \eqref{eq:var} (see also \cite{durastanti6}), and this consideration motivates the choice of the threshold $\thres$ in \eqref{eq:thres}. The true value of $s$ is unknown and, then, establishing explicitly $\Js$ is not possible. Nevertheless, the sum \eqref{eq:esti} truncated at $\Jn$ includes all the terms up to $\Js$, since 
		$$
		\Js \leq \Jn.
		$$		
		\end{remark}
		Finally, the next result establishes a lower bound for the rate of convergence, yielding optimality.  
		\begin{theorem}[Lower bound]\label{thm:maintheoremlower}
			Let $f \in \besovm\left(G\right)$, where $s-\frac{d}{r}>0$, let $\fest$ be defined by \eqref{eq:esti}, with $\eta$ given by \eqref{eq:thres} or \eqref{eq:thressoft}. Hence, if $1\leq p \leq \infty$
			\begin{equation*}\label{eq:maintheoremlower}
				\sup_{f \in \besov\left(G\right)} \Ex \left[\norm{\fest-\fm}_{\Lp}^p\right] \gtrsim n^{-\alpha\left(s,\mabs,p,r\right)},
			\end{equation*}
			where
			\begin{equation*}
			\alpha\left(s,\mabs,p,r\right)= \begin{cases}
			\frac{ps}{2\left(s+\mabs\right)+d} & \text{for }r\geq \frac{\left(2\mabs+d\right)p}{2\left(s+\mabs\right)+d}  \quad\text{(regular zone)}\\
			\frac{p\left(s+d\left(\frac{1}{p}-\frac{1}{r}  \right)\right)}{2\left[\left(s+\mabs\right)+d\left(\frac{1}{2}-\frac{1}{r} \right) \right]} & \text{for }r< \frac{\left(2\mabs+d\right)p}{2\left(s+\mabs\right)+d}  \quad\text{(sparse zone)}
			\end{cases},
			\end{equation*}
			as given by \eqref{eq:rate}. 
		\end{theorem}
\section{Numerical evidence}\label{sec:num}
In this section, we produce the results of some numerical experiments for $d=1,2$. We present the empirical evaluation of the $L^p$-risks ($p=2,3,5,\infty$), computed thanks to local thresholding techniques for the first and the second derivatives of some test density functions with respect to the choice of several values of the truncation bandwidth
 $\Jn$, the cardinality of cubature points $\Kj$ for any $j=1,\ldots,\Jn$, and the threshold constant $\kappa$. The weight function $b$ has been defined as the suitably rescaled primitive of the function $u \mapsto \exp(-\left(1-x^2\right)^{-1})$ (cf. Figure \ref{fig:01bis} and see also \cite{MaPeCUP}). Note that, as the simulations are produced over finite samples, they should be read as a reasonable hint.\\
	\begin{figure}\centering
	\includegraphics[width=0.6\textwidth]{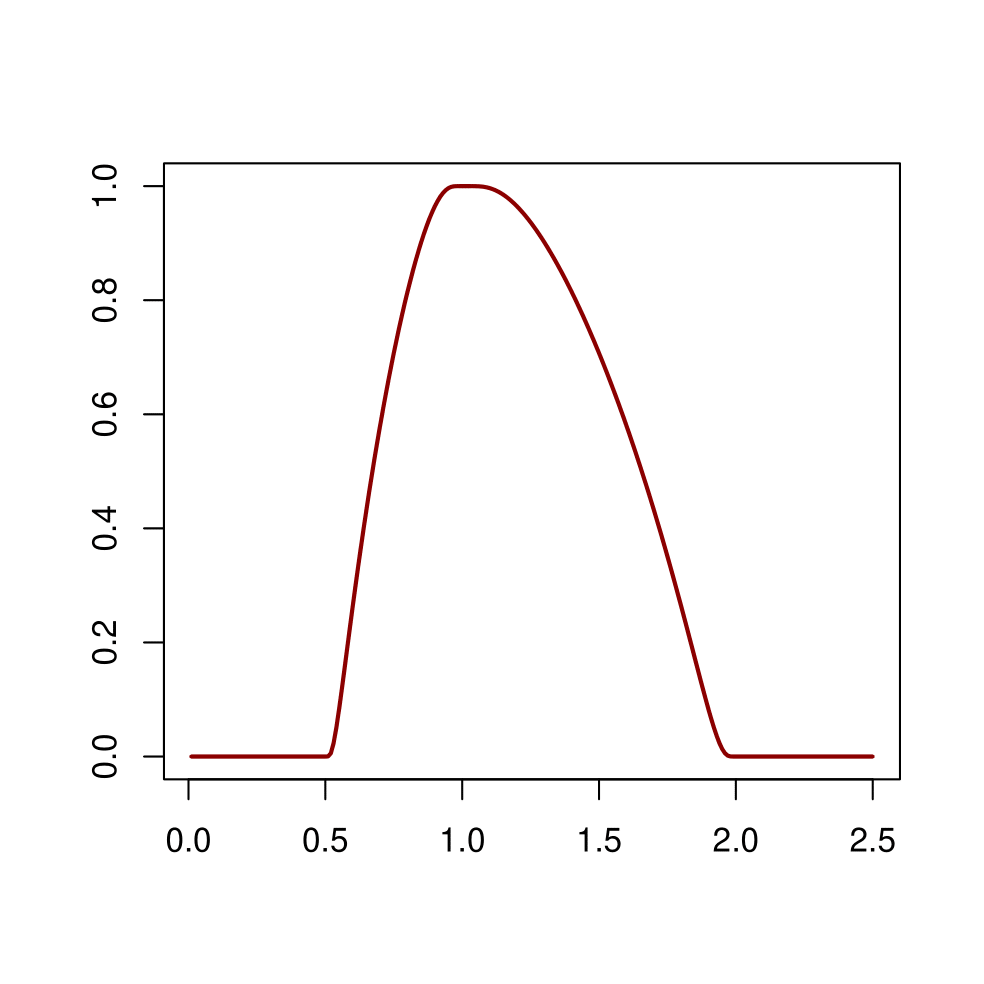}
	\caption{The needlet weight function $b$.}
	\label{fig:01bis}
\end{figure}
First, fixed $\N \in \naturals$, \eqref{eq:trunc} hints a value for truncation bandwidth $\Jn$. 
For any fixed $n,m$, let $L_\text{max}$ denote the highest eigenvalue for frequency indexes admitted within the frame $ \{ \psi_{j,k}^{(m)};j=1,\ldots,\Jn,k=1,\ldots,\Kj \}$ without being annihilated by the weight function $b$.   
For $\N = 2000, 10000$, we have 
$J_{2000,1}=3$ ($L_\text{max}=15$), $J_{2000,2}=2$ ($L_\text{max}=7$), for the first and the second derivatives respectively, and 
$J_{10000,1}=4$ ($L_\text{max}=31$), $J_{10000,2}=3$ ($L_\text{max}=15$).\\ 
Finally, concerning the thresholding constant $\kappa$, for which this paper only establishes a lower bound, we follow the lead given by \cite[Section 6]{bkmpAoSb}. For $q \in \naturals$, let us now define
$$
I_q = 
\int_{\frac{1}{B}}^{B} u^{q} b^{2}\left(u\right)\diff u;
$$ 
Chosen the function $b$, $I_{\mabs}/B^{j\mabs}$ provides an indication about the size of the square of the $L^2$-norm of $\psi_{j,k}^{(m)}$, since 
$$
\lim_{j\rightarrow 	\infty } \norm{\psi_{j,k}^{(m)}}_{L^2\left(\Td\right)}^2 = I_{\mabs}.
$$
In \cite[Section 6]{bkmpAoSb}, where for $m=0$, the authors have chosen on the sphere the analogous of $\kappa = \kappa_0 M I_0$, with $k_0>0$ and $M = \norm{f}_{L^\infty \left(\Td\right)}$, since using Cauchy-Schwarz inequality yields 
$$\left \vert \widehat{\beta}_{j,k}^{(0)}\right \vert = \int_{\Td} f(\theta) \bar{\psi}_{j,k}(\theta)\ \lesssim  M \norm{\psi_{j,k}}_{L^2\left(\Td\right)} \simeq M I_0.$$
In our case, for $d \in 1,2$ and $m\in\mathbb{N}_0^d$ such that $\mabs=1,2$, we choose  
$$
\kappa=\kappa_0 M I_{\mabs}, 
$$
trying different values for $k_0$.   
\paragraph {The uniform density.} We define on the unit circle ($d=1$) the following uniform density:
 $$f_{\operatorname{test};1}\left(\theta\right)=\frac{1}{2\pi}, \quad \theta \in \mathbb{T}^1, $$
so that 
$$
\betacm = 0 \quad \text{for any } j\in \naturals,k=1,\ldots,\Kj.$$ 
\\
In this case, counting the number of coefficients surviving the thresholding procedure gives an overview of the performance of our proposal. A measure of the goodness of the fit is obtained by estimating the $p$-powers of the $L^p$-risks, by means of the sums of the $p$-th powers of the selected coefficients.\\ 
For $\N=2000,10000$, all the empirical needlet coefficients ($m=0,1,2$) are discarded for $k_0=5.1$ and $k_0=3.8$.
Table \ref{tab:01} provides some evidence on the rate, collecting the estimates of the $p$ powers of the $L^p$-norms, $p=2,3,5,\infty$ for $k_0 = 0, 1, 2,3,4$, showing that the $L^p$-risk decreases along with $p$ and $\N$. Note that the case $k_0=0$ corresponds to the so-called linear estimator, where no selection on the needlet coefficients is performed. Finally, a comparison is provided with the estimator $\tilde{f}^{(m)}$, built by means of the standard harmonic basis on $\Td$ and defined by $\tilde{f}(\vartheta)=\sum_{\ell \in L_n} \tilde{a}_\ell^{(m)} S_\ell(\vartheta)$, where $$\tilde{a}^{(m)}_\ell=\frac{(-1)^m}{n}\sum_{i=1}^n \bar{S}_{\ell}^{(m)}(X_i)$$ and $L_n =\{\ell \in \mathbb{Z}^d: \abs{\ell} \leq B^{\Jn+1} \}$. Observe that, this choice of $L_n$ the $\tilde{f}$ and the needlet linear estimator are characterized by the same $\Lp$-risk, the latter being a linear combination of the former. 
As expected, for the highest value of $k_0$, no coefficient survives the thresholding procedure and then we find a higher value for all the risks. 
\begin{table}\centering
	\begin{tabular}{ |c|c||c|c|c|c|c| } 
		\hline
		
		\multicolumn{2}{|c||}{$L^p$-risks (power $p$)} & $\tilde{f}$ and $k_0=0$ & $k_0=1$ & $k_0=2$ & $k_0=3$& $k_0=4$ \\ 
		\hline
		\multicolumn{7}{|c|}{$n=2000$}\\
		\hline	
		\multirow{2}{*}{$p=2$}  & $m=1$ & $0.00841$ & $0.00722$  & $0.00296$ & $0.00249$  &  $0.00120$\\ 
		 & $m=2$ & $0.07470$  & $0.07350$  & $0.06129$ & $0.05879$ &  $0.05643$ \\  
		 \hline
		\multirow{2}{*}{$p=3$}  & $m=1$ & $1.23\cdot 10^{-3}$ & $1.04\cdot 10^{-3}$  & $2.91\cdot 10^{-4}$  & $2.43 \cdot 10^{-4}$ & $8.30 \cdot 10^{-5}$\\   
		&$m=2$ & $0.03366$  & $0.03282$  & $0.02850$ & $0.02798$ &  $0.02755$ \\ 
	  		\hline
		\multirow{2}{*}{$p=5$}  & $m=1$ & $4.19\cdot 10^{-5}$ & $3.37\cdot 10^{-5}$  & $3.82\cdot 10^{-6}$ & $2.99\cdot 10^{-6}$ & $5.09\cdot 10^{-7}$\\ 
		& $m=2$& $0.00947$  & $0.00908$  & $0.00822$ & $0.00814$  & $0.00779$ \\  
		\hline
		\multirow{2}{*}{$p=\infty$}  & $m=1$ & $0.26254$ & $0.26142$  & $0.25438$ & $0.20800$ & $0.06324$ \\
									&$m=2$ & $0.28402$  & $0.26142$  & $0.17695$ & $0.15665$ & $0.09593$ \\ 
		\hline							
		\multicolumn{7}{|c|}{$n=10000$}\\
		\hline	
		\multirow{2}{*}{$p=2$}  & $m=1$ & $0.01128$ & $0.00900$  & $0.00211$ & $0.00012$ & $0$\\
& $m=2$ & $ 0.11253$  & $0.09892$  & $0.07370$ & $0.01987$ & $0$ \\ 
\hline  
\multirow{2}{*}{$p=3$}  & $m=1$ &$1.67\cdot 10^{-3}$ & $1.28\cdot 10^{-3}$  & $1.94\cdot 10^{-4}$ & $4.07\cdot 10^{-7}$ & $0$\\
&$m=2$ & $0.05723$  & $0.04964$  & $0.03610$  & $0.00555$ & 0 \\  
\hline
\multirow{2}{*}{$p=5$}  & $m=1$ &$4.97\cdot 10^{-5}$ & $3.63\cdot 10^{-5}$  & $2.49\cdot 10^{-6}$ & $5.84\cdot 10^{-9}$ & $0$  \\ 
& $m=2$ &$0.02165$ & $0.01757$  & $0.01256$ & $0.00059$ & $0$  \\ 
\hline
\multirow{2}{*}{$p=\infty$}  & $m=1$ & $0.26721$ & $0.25019$  & $0.14390$ & $0.04776$ & $0$\\
&$m=2$ & $0.28172$  & $0.26501$  & $0.08387$ & $0.13909$ & $0$ \\
\hline 
	\end{tabular}
	\caption{Test density function $f_{\operatorname{test};1}$: some  $L^p$-risk evaluations ($p$-th powers) for $m=1,2$.}
	\label{tab:01}
\end{table}
\paragraph {The trigonometric density.} The next density function on the unit circle $d=1$ is defined as follows,
$$f_{\operatorname{test};2}\left(\theta\right)=\frac{1}{2\pi} \left(1+\frac{1}{2}\cos(2\theta)\right), \quad \theta \in \mathbb{T}^1. $$
This density features two local maxima (in $\pi/2$ and $3/2 \pi$) and two local minima (in $0$ and $\pi$). Using \eqref{eq:multider} yields
$$
\betacm = \cubew (-1)^m 2^m \bfun{\frac{2}{B^j}}\cos(2\xi_{j,k}) \quad \text{for any } j\in \naturals,k=1,\ldots,\Kj,$$ 
\\
which are not null only if $j=1$, in view of our choice of $b$. 
In this case, we aim to find an optimal threshold which allows to discard all the empirical coefficients with $j\neq 1$ but, at the same time, to keep the most relevant one with $j=1$ (corresponding to the peaks). In this sense, choosing a too high value for $k_0$ produces a worse $L^p$-risk estimate.\\ 
Table \ref{tabella:02} shows the estimates of the $p$-powers of the $L^p$-norms, $p=2,3,5,\infty$, for $k_0 = 2, 4,  8 , 25$ (again, the case $k_0=0$ is the linear estimator).
Table \ref{tab:03} collects the percentage of surviving coefficients. As expected, the better estimates for $L^p$-risks correspond to estimators where only coefficients with $j=1$ survive. 
\begin{table}\centering
	\begin{tabular}{ |c|c||c|c|c|c|c| } 
		\hline
		
		\multicolumn{2}{|c||}{$L^p$-risks} & $\tilde{f}$ and $k_0=0$ & $k_0=2$ & $k_0=4$ & $k_0=8$& $k_0=25$ \\ 
		\hline
		\multicolumn{7}{|c|}{$n=2000$}\\
		\hline	
		\multirow{2}{*}{$p=2$}  & $m=1$ & $2.05 \cdot 10^{-2}$ & $1.20 \cdot 10^{-2}$  & $2.04 \cdot 10^{-3}$ & $3.66 \cdot 10^{-5}$ & $1.26 \cdot 10^{-2}$ \\
		& $m=2$ &  0.03015 & 0.01800  & 0.00019 & 0.00019 & 0.05055 \\ 
		\hline 
		\multirow{2}{*}{$p=3$}  & $m=1$ &  $4.89 \cdot 10^{-3}$ &  $2.40 \cdot 10^{-3}$  &  $2.70 \cdot 10^{-4}$ &  $3.01 \cdot 10^{-7}$ &  $1.71 \cdot 10^{-3}$\\
		&$m=2$ & $7.41 \cdot 10^{-3}$  &  $3.59 \cdot 10^{-3}$ & $3.74 \cdot 10^{-6}$ & $3.74 \cdot 10^{-6}$ &  $1.36 \cdot 10^{-2}$\\
		\hline
		\multirow{2}{*}{$p=5$}  & $m=1$ & $4.27 \cdot 10^{-4}$ & $1.42 \cdot 10^{-4}$  & $6.22 \cdot 10^{-6}$ & $2.52 \cdot 10^{-11}$ & $3.47 \cdot 10^{-5}$ \\
		& $m=2$& $5.60 \cdot 10^{-4}$  & $1.84 \cdot 10^{-4}$  & $1.80 \cdot 10^{-9}$ & $1.80 \cdot 10^{-9}$ & $1.11 \cdot 10^{-3}$ \\ 
		\hline
		\multirow{2}{*}{$p=\infty$}  & $m=1$ & 0.40043& 0.31739  & 0.18048 & 0.01147 & 0.15915\\
		&$m=2$ & 0.35177  & 0.28271  & 0.02717 & 0.02717 & 0.31830 \\
		\hline	
		\multicolumn{7}{|c|}{$n=10000$}\\
		\hline	
		\multirow{2}{*}{$p=2$}  & $m=1$ & $1.82 \cdot 10^{-2}$ & $2.95 \cdot 10^{-3}$  & $4.17 \cdot 10^{-6}$ &  $4.17 \cdot 10^{-6}$ &  $1.26 \cdot 10^{-2}$ \\
		& $m=2$ &  $1.49 \cdot 10^{-1}$  &  $4.01 \cdot 10^{-2}$  &  $3.72 \cdot 10^{-5}$ &  $3.72 \cdot 10^{-5}$ &  $1.27 \cdot 10^{-2}$ \\ 
		\hline
		\multirow{2}{*}{$p=3$}  & $m=1$ & $3.92 \cdot 10^{-3}$ & $3.08 \cdot 10^{-4}$  & $1.19 \cdot 10^{-8}$ & $1.19 \cdot 10^{-8}$ & $1.71 \cdot 10^{-3}$ \\
		&$m=2$ & $8.71 \cdot 10^{-2}$  & $2.26 \cdot 10^{-2}$  & $3.18 \cdot 10^{-7}$ & $3.18 \cdot 10^{-7}$ & $1.36 \cdot 10^{-2}$ \\
		\hline
		\multirow{2}{*}{$p=5$}  & $m=1$ &  $2.62 \cdot 10^{-4}$  &  $4.57 \cdot 10^{-6}$   &  $1.22 \cdot 10^{-13}$  & $1.22 \cdot 10^{-13}$ &  $3.47 \cdot 10^{-5}$ \\
		& $m=2$& $4.52 \cdot 10^{-2}$  & $9.82 \cdot 10^{-3}$  & $2.89 \cdot 10^{-11}$ & $2.89 \cdot 10^{-11}$ & $1.10 \cdot 10^{-3}$ \\ 
		\hline 
		\multirow{2}{*}{$p=\infty$}  & $m=1$ & $0.38401$ & $0.15848$  & $0.00400$ & $0.00400$ & $0.15915$ \\
		&$m=2$ & $1.03924$  & $0.82377$  & $0.01186$ & $0.01186$ & $0.31830$ \\ 
		\hline
	\end{tabular}
	\caption{Test density function $f_{\operatorname{test};2}$: some $L^p$-risk evaluations, for $m=1$ (first entry) and $m=2$ (second entry) }
	\label{tabella:02}
\end{table}
\begin{table}\centering
	\begin{tabular}{ |c|c||c|c|c|c|c| } 
		\hline
		 \multicolumn{2}{|c|}{Surviving coefficients (\%)}	 & $k_0=0$ & $k_0=2$ & $k_0=4$ & $k_0=8$& $k_0=25$ \\ 
		\hline
				\multicolumn{7}{|c|}{$n=2000$}\\
		\hline
				\multirow{2}{*}{$j=0$} & $m=1$ & 100  & 50 & 0 & 0 & 0 \\ 
		& $m=2$ & 100  & 50  & 0 & 0 & 0 \\		\hline
		\multirow{2}{*}{$j=1$} & $m=1$ & 100 & 50  & 50  & 50 &  0  \\ 
		& $m=2$ & 100  & 50  & 50 & 50 & 0\\
				\hline
		\multirow{2}{*}{$j=2$} & $m=1$ & 100& 12.5  & 0  & 0 &  0  \\ 
		& $m=2$ & 100  & 31.25  & 0 & 0 & 0\\		\hline
		\multirow{2}{*}{$j=3$} & $m=1$ &100 & 21.875  & 3.125   & 0 &  0  \\ 
		& $m=2$ & NA  & NA  & NA & NA & NA\\
		\hline
				\multicolumn{7}{|c|}{$n=10000$}\\ 
		\hline
		\multirow{2}{*}{$j=0$} & $m=1$ & 100  & 50 & 0 & 0 & 0 \\ 
		& $m=2$ & 100  & 50  & 0 & 0 & 0 \\		\hline
		
		\multirow{2}{*}{$j=1$} & $m=1$ & 100  & 50 & 50 & 50 & 0 \\ 
		& $m=2$ & 100  & 62.50  & 50 & 50 & 0 \\		\hline
		\multirow{2}{*}{$j=2$} & $m=1$ & 100  & 0  & 0 & 0 & 0 \\ 
		& $m=2$ & 100  &  12.50 & 0 & 0 & 0\\		\hline
		\multirow{2}{*}{$j=3$} & $m=1$ & 100  & 6.25  & 0 & 0 & 0 \\ 
		& $m=2$ & 100  &  6.25 & 0 & 0 & 0\\ 		\hline
		\multirow{2}{*}{$j=4$} & $m=1$ & 100  & 6.25  & 0 & 0 & 0 \\ 
		& $m=2$ & NA  & NA  & NA & NA & NA \\	
		\hline
		\end{tabular}
		\caption{Test density function $f_{\operatorname{test};2}$: percentage of surviving coefficients, for first and second derivatives at $k_0=0,2,4,8,25$.}
		\label{tab:03}
\end{table}
\paragraph{Wrapped normal distribution.}
\begin{figure}
\centering
\begin{subfigure}[b]{0.32\textwidth}
	\centering
	\includegraphics[width=\textwidth]{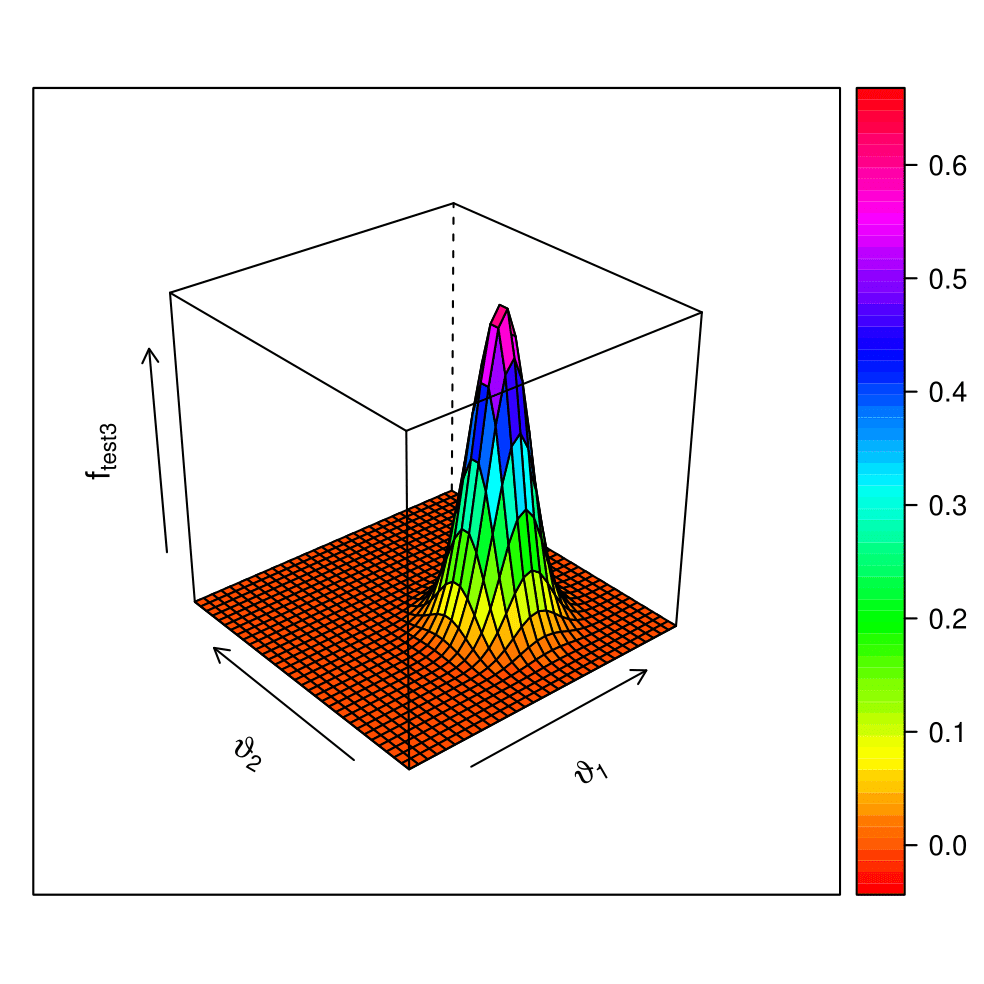}
	\caption{$f_{\operatorname{test};3}$}
	\label{fig:02a}
\end{subfigure}
\hfill
\begin{subfigure}[b]{0.32\textwidth}
	\centering
	\includegraphics[width=\textwidth]{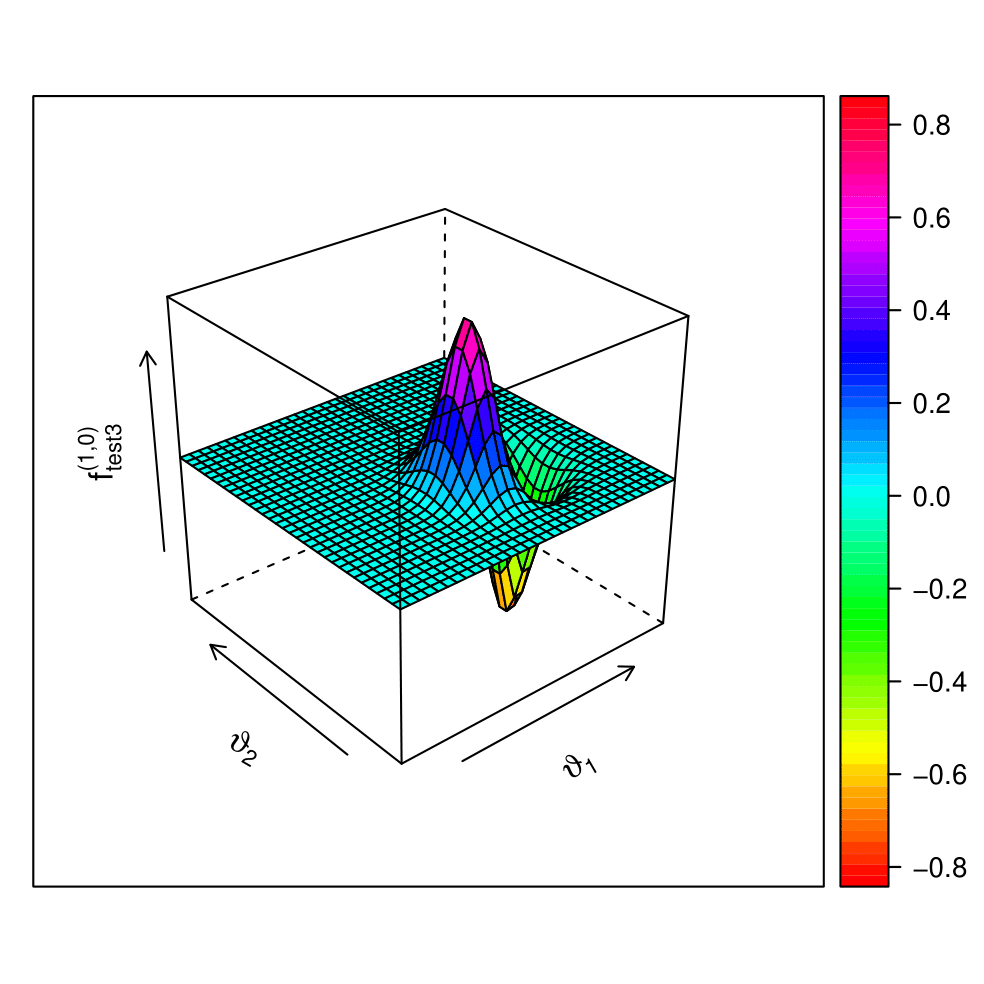}
	\caption{$f_{\operatorname{test};3}^{(1,0)}$}
	\label{fig:02b}
\end{subfigure}
\hfill
\begin{subfigure}[b]{0.32\textwidth}
	\centering
	\includegraphics[width=\textwidth]{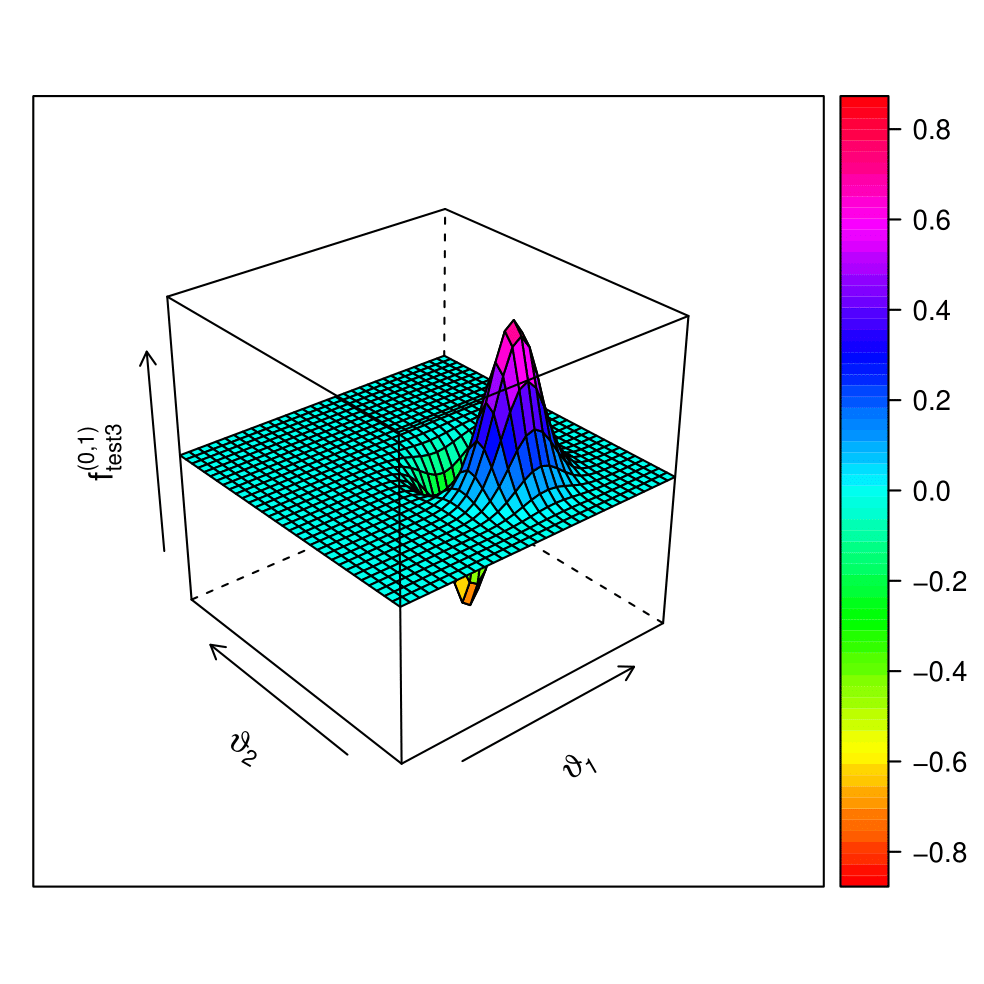}
	\caption{$f_{\operatorname{test};3}^{(0,1)}$}
	\label{fig:02c}
\end{subfigure}
\caption{The test density $f_{\operatorname{test};3}$ and its first derivatives}
\label{fig:02}
\end{figure}
Here we consider the density function of a $2$-d wrapped normal distribution, defined as:
$$f_{\operatorname{test};3}\(\theta_1,\theta_1\) = \frac{1}{2\pi \sigma_1\sigma_2}\sum_{k_1=-100}^{100}\sum_{k_2=-100}^{100} e^{-\frac{\(\theta_1-\mu_1+2\pi k_1\)^2}{2\sigma_1^2}-\frac{\(\theta_1-\mu_1+2\pi k_2\)^2}{2\sigma_1^2}}, \quad \left(\theta_1,\theta_2\right) \in \mathbb{T}^2. $$
We choose $\mu_1=3/4 \pi$, $\mu_2=5/4 \pi$, $\sigma_1=\sigma_2=0.5$. Figure \ref{fig:02} plots the graphs of $f_{\operatorname{test};3}$ and both its derivatives. Our goal is to study the two first derivatives of  $f_{\operatorname{test};3}$ testing several values of $k_0$. The goodness of our choice is evaluated thanks to the estimates of the $L^p$-risks as described above for $n=10000$. 
Table \ref{tab:04} contains the estimates of the $L^p$-norm, $p=2,3,5,\infty$ for $k_0 = 0, 0.5, 1, 8,15$. Some of the estimators here obtained are shown in Figure \ref{fig:03}. The optimal choice of $k$ lies around $k_0=1$, for which we obtain better $L^p$-risks. As shown in Figures \ref{fig:04b},\ref{fig:04e} the graph is evidently smoother and less perturbed than the linear estimator ($k_0=0$) (see Figures \ref{fig:04a},\ref{fig:04b}).
For $k_0=15$, the peak is lower than the theoretical one, so that smaller fluctuations appear more evidently also in the graph (Figures \ref{fig:04c},\ref{fig:04f}). 
\begin{table}\centering
	\begin{tabular}{ |c|c||c|c|c|c|c| } 
		\hline
		\multicolumn{2}{|c||}{$L^p$-risks, $n=10000$} & $\tilde{f}$ and $k_0=0$ & $k_0=0.5$ & $k_0=1$ & $k_0=8$& $k_0=15$ \\ 
		\hline
		\multirow{2}{*}{$p=2$}  & $m=(1,0)$ & 0.00806 & 0.00699  & 0.00485 & 0.00601 & 0.00781 \\
		& $m=(0,1)$ &  0.00889 & 0.00711  & 0.00493 & 0.00621  & 0.00801 \\ 
		\hline 
		\multirow{2}{*}{$p=3$} &$m=(1,0)$ & 0.00312 &  0.00227 &  0.00100  &    0.00102 &   0.00216 \\
		&$m=(0,1)$ & 0.00343  &  0.00271 & 0.00103 & 0.00110 &  0.00221\\ 
		\hline
		\multirow{2}{*}{$p=5$}  & $m=(1,0)$ & $1.05 \cdot 10^{-3}$ & $5.65 \cdot 10^{-4}$  & $9.53 \cdot 10^{-5}$ & $5.12 \cdot 10^{-5}$ & $3.42 \cdot 10^{-4}$ \\
		& $m=(0,1)$& $2.07 \cdot 10^{-3}$  & $6.24 \cdot 10^{-4}$  & $1.20 \cdot 10^{-4}$ & $9.81 \cdot 10^{-5}$ & $4.11 \cdot 10^{-4}$ \\ 
		\hline
		\multirow{2}{*}{$p=\infty$}  & $m=(1,0)$ & 0.94465& 0.80713  & 0.36937 & 0.43902 & 0.57350\\
		&$m=(0,1)$ & 0.95493  & 0.82432  & 0.36937 & 0.47981 & 0.57483 \\
		\hline	
		\hline
	\end{tabular}
	\caption{Test density function $f_{\operatorname{test};3}$: some $L^p$-risk evaluations for the first derivatives}
	\label{tab:04}
\end{table}
\begin{figure}
	\centering
	\begin{subfigure}[b]{0.32\textwidth}
		\centering
		\includegraphics[width=\textwidth]{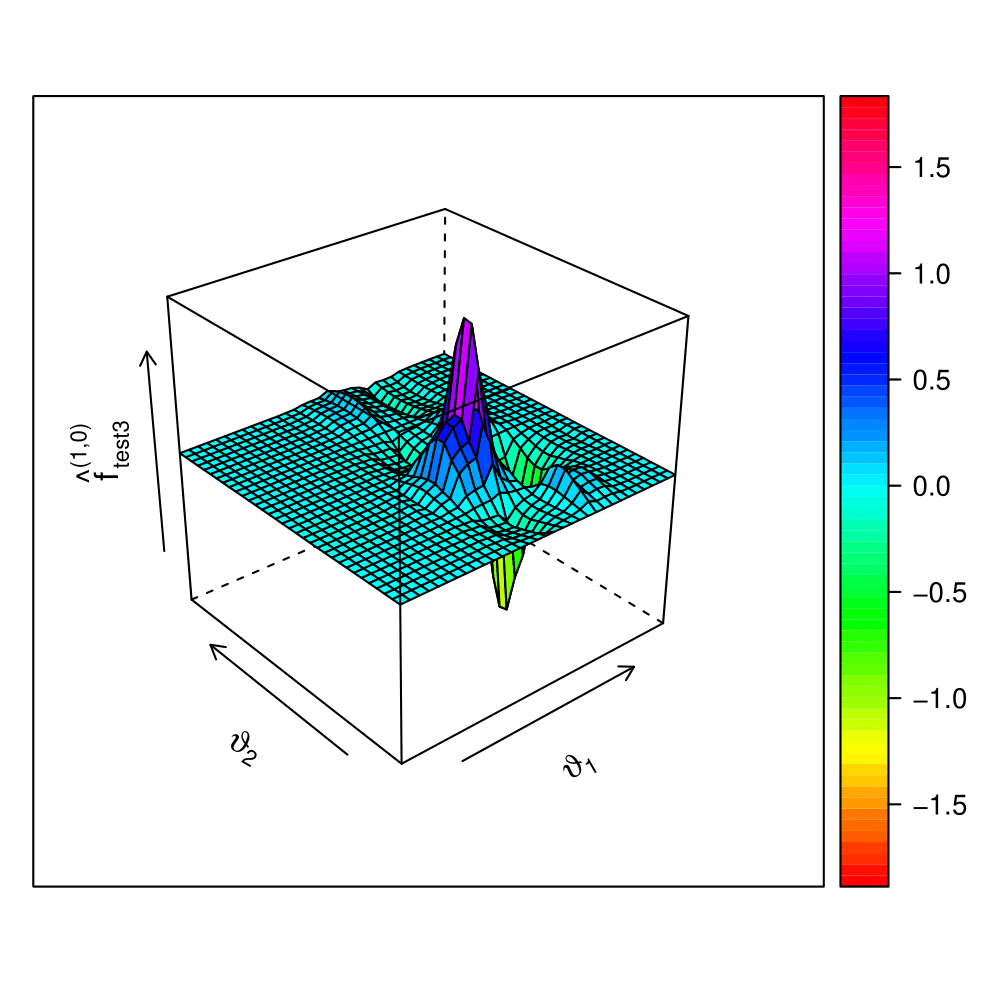}
		\caption{$k_0=0$}
		\label{fig:04a}
	\end{subfigure}
	\hfill
	\begin{subfigure}[b]{0.32\textwidth}
		\centering
		\includegraphics[width=\textwidth]{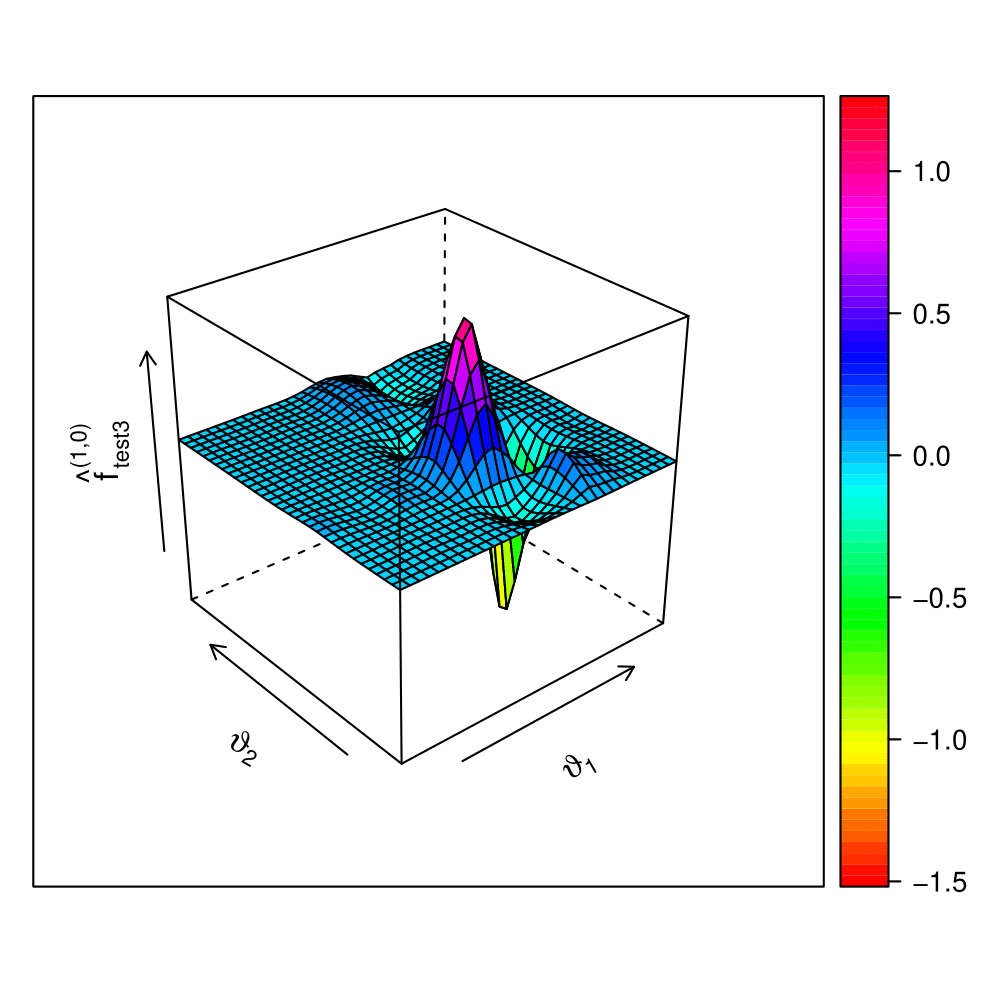}
		\caption{$k_0=1$}
		\label{fig:04b}
	\end{subfigure}
	\hfill
	\begin{subfigure}[b]{0.32\textwidth}
		\centering
		\includegraphics[width=\textwidth]{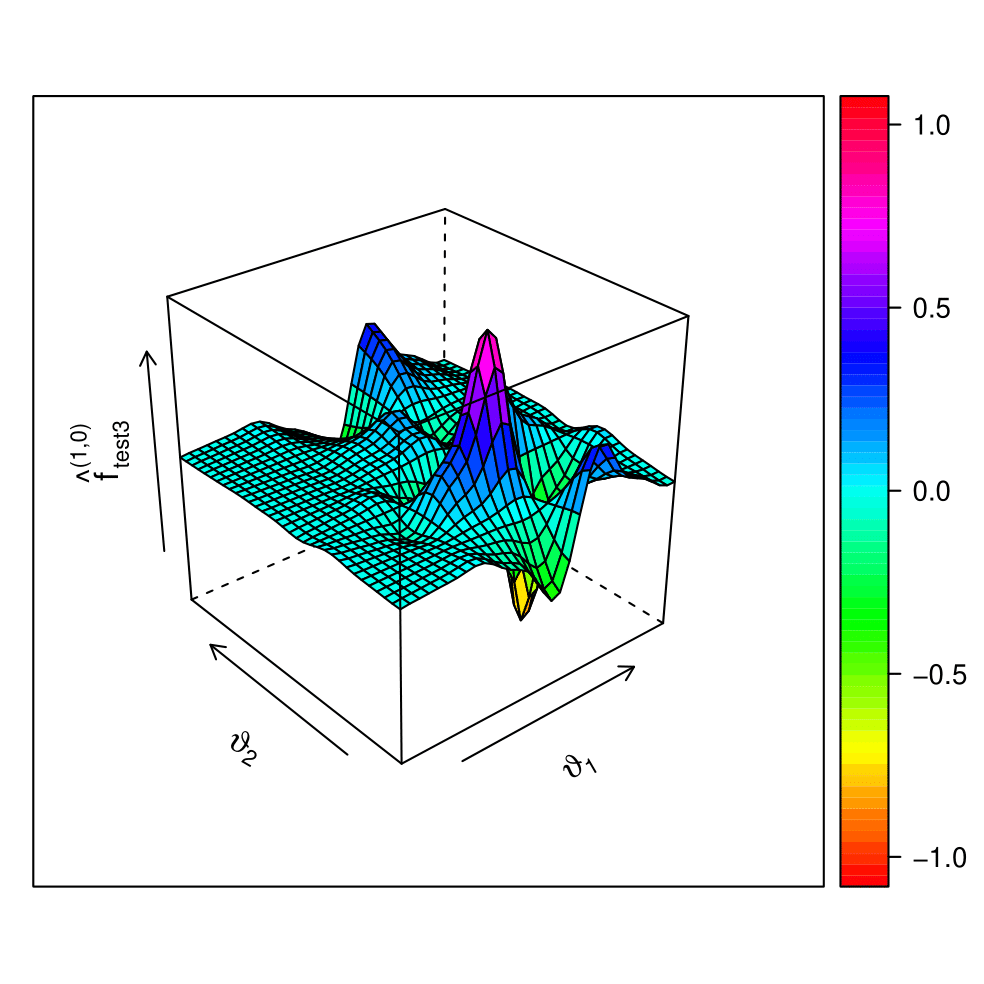}
		\caption{$k_0=15$}
		\label{fig:04c}
	\end{subfigure}
	\begin{subfigure}[b]{0.32\textwidth}
	\centering
	\includegraphics[width=\textwidth]{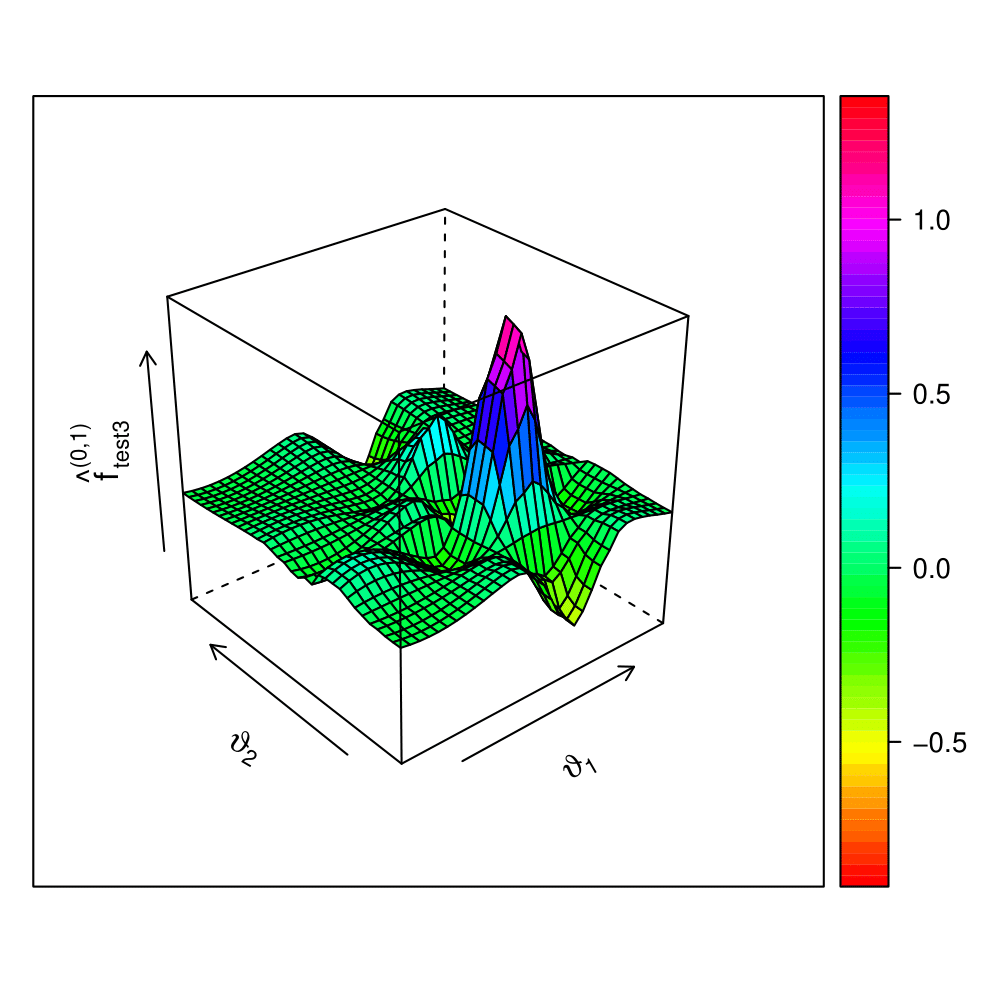}
	\caption{$k_0=0$}
	\label{fig:04d}
\end{subfigure}
\hfill
\begin{subfigure}[b]{0.32\textwidth}
	\centering
	\includegraphics[width=\textwidth]{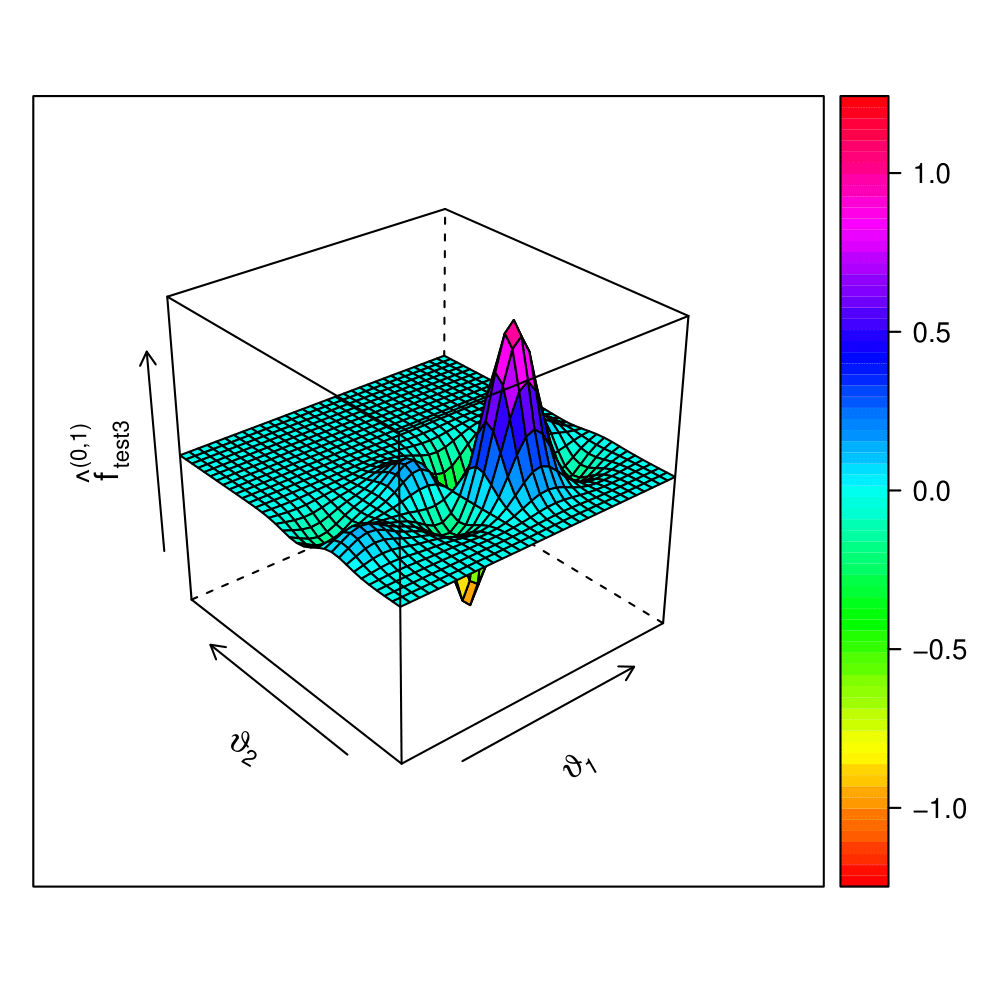}
	\caption{$k_0=1$}
	\label{fig:04e}
\end{subfigure}
\hfill
\begin{subfigure}[b]{0.32\textwidth}
	\centering
	\includegraphics[width=\textwidth]{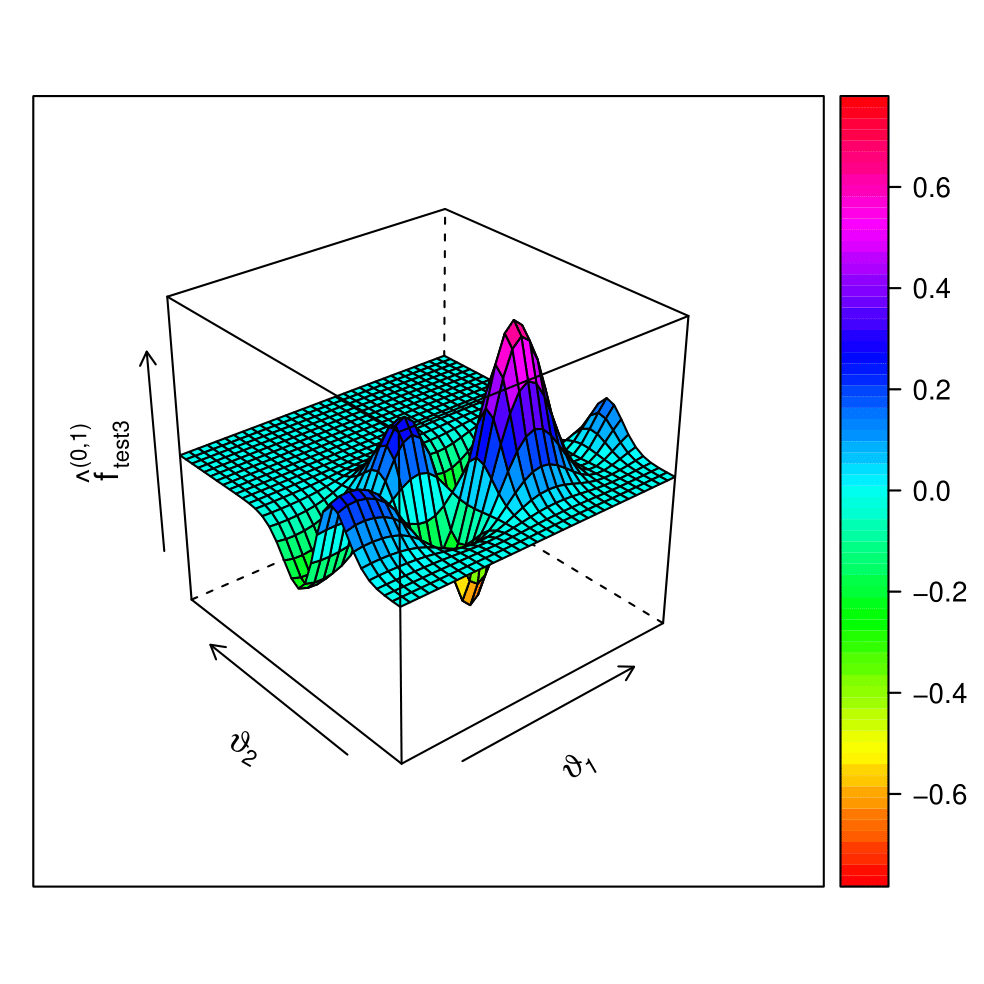}
	\caption{$k_0=15$}
	\label{fig:04f}
\end{subfigure}
	\caption{The density estimators $\hat{f}_{\operatorname{test};3}^{(1,0)}$ (top panels), and $\hat{f}_{\operatorname{test};3}^{(0,1)}$ (bottom panels), for $k_0=0,1,15$ }
	\label{fig:03}
\end{figure}

The interaction between multiresolution and the threshold technique shows here its advantages: finer levels of resolution are non discarded only at locations where they the theoretical curve is less regular. Hence, the multiresolution properties of the thresholding needlet estimator allow for local adaptation if the considered derivative presents peaks and different slopes.

\section{Proofs}\label{sec:proof}
\subsection{Auxiliary results on needlet derivatives and Besov spaces}\label{sec:needproof}
\begin{proof}[Proof of Lemma \ref{lemma:derivative}] First of all, observe that, for any $i=1,\ldots,d$ and for any $n \in \integers$, it holds that
	\begin{equation*}
	\frac{\partial^{m_i}}{\partial \theta_{i}^{m_i}} \Sn\(\theta\)=\(-1\)^{m_i}\ell_i^{m_i} \Sn\(\theta\).
	\end{equation*}
	Thus, 
	$$
	\Dm \Sn \(\theta\)= \(-1\)^{\mabs}\prod_{i=1}^{d} \ell_i^{m_i} \Sn\(\theta\).
	$$
	Using \eqref{eq:multider} yields
	\begin{align*}
	\Dm \psijkm\(\theta\) &= \sqrt{\cubew} \sumn \bfun{\barg} \Snc \(\cubep\) \Dm \Sn\(\theta\)\\
	&= \sqrt{\cubew} \sumn \bfun{\barg}  \(-1\)^{\mabs}\prod_{i=1}^{d} \ell_i^{m_i} \Snc \(\cubep\) \Sn\(\theta\)\\
	&= \sqrt{\cubew} B^{j\mabs} \sumn \bfunm{\ell} \Snc \(\cubep\) \Sn\(\theta\),
	\end{align*}
	as claimed.
\end{proof}
\begin{proof}[Proof of Lemma \ref{lemma:sumbeta}]
	First of all, observe that \eqref{eq:sumbeta1} is obtained by integrating iteratively \eqref{eq:derrec} by parts, thanks to the periodicity of the function $f$ over $\Td$. 
	As far as \eqref{eq:sumbeta2} is concerned, for any $\theta \in \Td$, observe first of all that
	\begin{align*}
	\sum_{k=1}^{\Kj}\psijk \(\theta\) \cc{\psijkm}\(\theta^\prime\) & = \sum_{k=1}^{\Kj} \cubew  B^{j\mabs} \sumn \sum_{\npr \in \Lambdaj} \bfunm{\ell} b_j^{\(0\)}\(\npr\) \cc{\Sn}\(\theta^\prime\) \Sn\(\cubep\)\cc{\Snprime}\(\cubep\)\Snprime\(\theta\) \\
	& =  \(-1\)^{\mabs} \sumn\( \prod_{i=1}^d \ell_i^{m_i} \) b^2\(\barg\) \Sn\(\theta-\theta^\prime\).
	\end{align*}
	Thus, it holds that
	\begin{align*}
	\sum_{k=1}^{\Kj}\betacm \psijk\(\theta\) & =  \int_{\Td} f\(\theta^\prime\) \sumk \cc{\psijkm}\(\theta^\prime\)\psijk\(\theta\) \rho\(\diff \theta^\prime \)\\  
	&=  \(-1\)^{\mabs} \sumn\( \prod_{i=1}^d \ell_i^{m_i} \) b^2\(\barg\) \Sn\(\theta\) \int_{\Td} f\(\theta^\prime\) \Snc\(\theta^\prime\) \rho\(\diff \theta^\prime\) \\
	&=\sum_{k=1}^{\Kj} \cubew  B^{j\mabs} \sumn \sum_{\npr \in \Lambdaj} \bfunm{\ell} b_j^{\(0\)}\(\npr\) \angles{f,\Snprime}_{\Ltwo}  \Snprime\(\cubep\)\cc{\Sn}\(\cubep\)\Sn\(\theta\)\\
	& =\sum_{k=1}^{\Kj}\betac \psijkm\(\theta\), 
	\end{align*}	  
	as claimed.
\end{proof}
	\subsection{Auxiliary probabilistic results}
		\begin{proof}[Proof of Lemma \ref{lemma:probbounds}] 
			This proof follows strictly the one of \cite[Lemma 16]{bkmpAoSb}, see also \cite{dmg}.\\
			Observe that \eqref{eq:bernstein} is obtained by means of the Bernstein's inequality (see, for example, \cite{WASA}), applied to the set $\{\cc{\psijkm}\(X_1\),\ldots,\cc{\psijkm}\(X_{\N}\)\}$, observing additionally that 
			$$
				\Var	 	\left(\cc{\psijkm}\(X_i\)\right)\leq\Ex\left[\(\cc{\psijkm}\(X_i\)\)^2\right]\leq  M \norm{\psijkm}_{\Ltwo} ^2\leq M C_2^\ast B^{2j\mabs}, 
			$$	
			while 
			$$
				\abs{\cc{\psijkm}\left(X_i\right)} \leq \norm{\psijkm}_{\Linf}\leq C^\ast_\infty B^{j\left(\mabs+\frac{d}{2}\right)}\leq C^\ast_\infty B^{j\mabs}\sqrt{\N},
			$$
			since $B^{dj}\leq\N$.\\
			Now, in order to prove \eqref{eq:momenteta} and \eqref{eq:momentinf}, first observe that, for $x >0$,
			\begin{align*}
				\Pr\left(\abs{\betamest-\betacm}\geq x \right)\leq 2\left(\exp\left(-\frac{\N x^2}{4 B^{j\mabs}C^{\ast}_2 M }\right) \right) + \exp\left(-\frac{3\sqrt{\N}x }{4 B^{j\mabs} C^{\ast}_{\infty}  }\right).
			\end{align*} 
			Hence, on the one hand, for $\eta \geq 1$, we have
			\begin{align*}
				\Ex\left[\abs{\betamest-\betacm}^\eta\right] &=\int_{0}^{\infty} x^{\eta-1}\Pr\left(\abs{\betamest-\betacm}\geq x \right) \diff x\\
				& \leq 2 \int_{0}^{\infty} x^{\eta-1}\left(\exp\left(-\frac{\N x^2}{4 B^{j\mabs}C^{\ast}_2 M }\right) \right) + \exp\left(-\frac{3\sqrt{\N}x }{4 B^{j\mabs} C^{\ast}_{\infty}  }\right) \diff x\\ 
				& \lesssim \N^{-\frac{\eta}{2}} B^{j\mabs\frac{\eta}{2}}
			\end{align*}
			using 
			 $B^{j\mabs}\leq\N$. On the other hand, let us preliminarily fix $$a =\frac{1}{\(\log B\)} \max \left(\frac{8}{3}dC^\ast_{\infty},2\sqrt{2}d C_2^\ast M\right)$$ and then write
			\begin{align*}
				\Ex\left[\sup_{k=1,\ldots,\Kj}\abs{\betamest-\betacm}^\eta\right] &\leq\int_{0}^{\frac{ajB^{j\frac{\mabs}{2}}}{\sqrt{\N}}} x^{\eta-1}\diff x \\
				&+2\int_{\frac{ajB^{j\frac{\mabs}{2}}}{\sqrt{\N}}}^{\infty}x^{\eta-1}\Pr\left(\sup_{k=1,\ldots,\Kj}\abs{\betamest-\betacm}\geq x \right)\!\diff x\\
				& \leq 2 \int_{0}^{\infty} \! x^{\eta-1}\left(\exp\left(-\frac{\N x^2}{4 B^{j\mabs}C^{\ast}_2 M }\right) \right) + \exp\left(-\frac{3\sqrt{\N}x }{4 B^{j\mabs} C^{\ast}_{\infty}  }\right) \!\diff x.
			\end{align*}
			If $x\geq ajB^{j\frac{\mabs}{2}}/\sqrt{\N}$, we have
			$$
				B^{jd}e^{-\frac{-\N x^2}{4B^{j\mabs} C_2^\ast M}}= e^{-\frac{-\N x^2}{8B^{j\mabs} C_2^\ast M}-\frac{-\N x^2}{8B^{j\mabs} C_2^\ast M}+\(\log B\) dj}\leq  e^{-\frac{- \N x^2}{8B^{j\mabs} C_2^\ast M}}, 
			$$
			and
			$$
				B^{jd}e^{-\frac{-3\sqrt{\N} x}{4 B^{j\mabs} C_\infty^\ast}}= e^{-\frac{-3\sqrt{\N} x}{8 B^{j\mabs} C_\infty^\ast}  -\frac{-3\sqrt{\N} x}{8 B^{j\mabs} C_\infty^\ast} + \(\log B\) jd}\leq e^{-\frac{-3\sqrt{\N} x}{8 B^{j\mabs} C_\infty^\ast}},
			$$ 
			so that
			\begin{align*}
				\Ex\left[\sup_{k=1,\ldots,\Kj}\abs{\betamest-\betacm}^\eta\right] & \lesssim \left(\frac{j+1}{\sqrt{\N}}\right)^{\eta},
			\end{align*}
		as claimed.
		\end{proof}
	\begin{proof}[Proof of Lemma \ref{lemma:sussig}]
		The proof follows directly \eqref{eq:bernstein}, choosing $x=\thres$ and observing that 
		\begin{align*}
			\exp\(- \frac{3\kappa^2 B^{j \mabs } \N}{4C_\infty ^\ast\(6\frac{C_2^\ast M}{C_\infty ^\ast} + \kappa \)}\) & \leq \exp\(- \frac{3\kappa B^{j \mabs } \N}{8C_\infty^\ast}\)\\
			& \lesssim \N^{-\frac{3}{8C_\infty^\ast} \kappa}\\
			& \lesssim \N^{-\frac{\gamma}{2}},			
		\end{align*}
		given that 
		$$
			\gamma \leq \left(\frac{3}{8C_\infty^\ast}\right)\kappa, 
		$$
	as claimed.
	\end{proof}
\subsection{Main results}
Let us start by proving the upper bound.
\begin{proof}[Proof of Theorem \ref{thm:maintheorem}]
	Let us consider $p<\infty$. 
As far as $\Sigma$ is concerned, properly adapting to our problem the procedure presented in \cite{bkmpAoSb} (see also \cite{donoho1,dmg}) and using \eqref{eq:thres} yield
\begin{align*}
&\Sigma /  \Jn^{p-1} \leq\sumjtrunc \norm{\sumk \Ex \left[\left(\eta\left(\betamest, \thres\right)-\betacm\right)\psijk\right]}_{\Lp}^p\\
&=\!\sumjtrunc \Ex\left[ \norm{\sumk \left(\eta\left(\betamest, \thres\right)-\betacm\right)\psijk}_{\Lp}^p\!\!\!\!\!\!\ind{\abs{\betamest}\geq \thres} \ind{\abs{\betacm}\geq \frac{1}{2}\thres}\right]  \\
& + \!\sumjtrunc \Ex\left[ \norm{\sumk \left(\eta\left(\betamest, \thres\right)-\betacm\right)\psijk}_{\Lp}^p\!\!\!\!\!\!\ind{\abs{\betamest}\geq \thres}\ind{\abs{\betacm}< \frac{1}{2}\thres}\right]\\
& +\! \sumjtrunc \Ex\left[ \norm{\sumk \left(\eta\left(\betamest, \thres\right)-\betacm\right)\psijk}_{\Lp}^p\!\!\!\!\!\! \ind{\abs{\betamest}< \thres}\ind{\abs{\betacm}\geq \frac{1}{2}\thres}\right]  \\
& + \!\sumjtrunc \Ex\left[ \norm{\sumk \left(\eta\left(\betamest, \thres\right)-\betacm\right)\psijk}_{\Lp}^p\!\!\!\!\!\!\ind{\abs{\betamest}< \thres}\ind{\abs{\betacm}< \frac{1}{2}\thres}\right]\\
&\eqqcolon Aa+Au+Ua+Uu,
\end{align*}
where the regions considered are the analogous of $Bb$, $Bs$, $Sb$, $Ss$ in \cite{bkmpAoSb}. As in \cite{dmg,durastanti6}, we prefer to modify this notation since $B$, $b$, and $s$ are already used in the current work. While the upper bounds for $Ua$ and $Uu$ are the same for hard and soft thresholding, we have to follow two slightly different approach to bound $Au$ and $Ua$. Furthermore, by the heuristic  point of view, the bounds two cross/terms $Au$ and $Ua$ depend on the inequality \eqref{eq:sussig} in Lemma \ref{lemma:sussig}. Indeed, in the hard thresholding settings, using Cauchy--Schwarz inequality and Equations \eqref{eq:momenteta} and \eqref{eq:sussig} leads to 
\begin{align}
\notag Au & \lesssim \sumjtrunc\Ex\left[ \norm{\sumk \left(\betamest-\betacm\right)\psijk}_{\Lp}^p\ind{\abs{\betamest-\betacm}\geq \frac{\thres}{2}}\right]\\
\notag&\lesssim  \sumjtrunc B^{jd\left(\frac{p}{2}-1\right)}\sumk \Ex\left[  \abs{\betamest-\betacm}^p\ind{\abs{\betamest-\betacm}\geq \frac{\thres}{2}}\right]\\
\notag&\lesssim\sumjtrunc B^{jd\left(\frac{p}{2}-1\right)}\sumk \Ex^{\frac{1}{2}}\left[  \abs{\betamest-\betacm}^{2p}\right] \Pr^{\frac{1}{2}}\left(\abs{\betamest-\betacm}\geq \frac{\thres}{2}\right)\\
\notag&\lesssim \Jn \N^{-\frac{p}{2}}\N^{-\frac{\gamma}{2}} \sumjtrunc B^{\frac{p\Jn}{2}\left(d+2\mabs\right)}\\
\label{eq:Au}&\leq  \N^{-\frac{p+\gamma}{2}} \left(\frac{\N}{\log \N}\right)^{\frac{p}{2}}
\leq \left(\log \N\right)^{-\frac{p}{2}} \N^{-\frac{\gamma}{2}},
\end{align}
while in the soft thresholding framework, using additionally the triangle inequality yields
\begin{align}
\notag Au & 
\lesssim\sumjtrunc B^{jd\left(\frac{p}{2}-1\right)}\sumk \Ex^{\frac{1}{2}}\left[  \abs{\betamest-\betacm-\thres}^{2p}\right] \Pr^{\frac{1}{2}}\left(\abs{\betamest-\betacm}\geq \frac{\thres}{2}\right)\\
\notag&\lesssim\sumjtrunc B^{jd\left(\frac{p}{2}-1\right)}\sumk \left(\Ex\left[  \abs{\betamest-\betacm}^{2p}\right]+\thres^{2p}\right)^{\frac{1}{2}} \Pr^{\frac{1}{2}}\left(\abs{\betamest-\betacm}\geq \frac{\thres}{2}\right)\\
\notag&\lesssim \Jn \N^{-\frac{p}{2}}\N^{-\frac{\gamma}{2}} \sumjtrunc B^{\frac{p\Jn}{2}\left(d+2\mabs\right)}\\
&\leq  \N^{-\frac{p+\gamma}{2}} \left(\frac{\N}{\log \N}\right)^{\frac{p}{2}}
\leq \left(\log \N\right)^{-\frac{p}{2}} \N^{-\frac{\gamma}{2}}\label{eq:Ausoft}.
\end{align}
Note that, in both the cases, for \eqref{eq:Au} and \eqref{eq:Ausoft} we have derived the same bound.
As far as the other cross/term is considered, in both the cases, it holds that
\begin{align}
	\notag Ua & \lesssim  \sumjtrunc \norm{\sumk \betacm \psijk}_{\Lp}^p \Pr\left({\abs{\betamest-\betacm}> \thres}\right)\\
	&\lesssim \norm{\fm}_{\Lp}^p \N^{-\gamma}.\label{eq:Ua}
\end{align}
As far as $Aa$, and $Uu$ are concerned, we derive their upper bounds by using the tail behavior in the Besov balls $\besov\(G\)$, the crucial role of the optimal bandwidth selection $\Js$ defined by \eqref{eq:optimalselection}, and the bound on the centered moments of $\betamest$ given in Lemma \ref{lemma:probbounds}. Indeed, in the hard thresholding settings, using \eqref{eq:momenteta}, we obtain that   
\begin{align*}
	Aa& \lesssim \sumjtrunc \Ex\left[\norm{\sumk\left(\betamest-\betacm\right)\psijk}_{\Lp}^p\right] \ind{\abs{\betacm}\geq \frac{1}{2}\thres}\\
	&= \sumjtrunc \sumk \Ex\left[\abs{\betamest-\betacm}^p\right]\norm{\psijk}_{\Lp}^p \ind{\abs{\betacm}\geq \frac{1}{2}\thres}\\
	&\lesssim 	\N^{-\frac{p}{2}}\sumjtrunc B^{jp\mabs}\sumk\norm{\psijk}_{\Lp}^p \ind{\abs{\betacm}\geq \frac{1}{2}\thres}
\end{align*}	
In the soft thresholding framework, similarly to \eqref{eq:Ausoft}, using the triangular inequality leads to 
\begin{align*}
Aa& \lesssim \sumjtrunc \sumk \Ex\left[\abs{\betamest-\betacm-\thres}^p\right]\norm{\psijk}_{\Lp}^p \ind{\abs{\betacm}\geq \frac{1}{2}\thres}\\
& \lesssim \sumjtrunc \sumk \left(\Ex\left[\abs{\betamest-\betacm}^p\right]+\thres^p\right)\norm{\psijk}_{\Lp}^p \ind{\abs{\betacm}\geq \frac{1}{2}\thres}\\
&\lesssim 	\N^{-\frac{p}{2}}\sumjtrunc B^{jp\mabs}\sumk\norm{\psijk}_{\Lp}^p \ind{\abs{\betacm}\geq \frac{1}{2}\thres}
\end{align*}	
Now, using the definition of the optimal bandwidth selection \eqref{eq:optimalselection}, we can rewrite in both the cases $Aa$ as the sum of two finite series, that is, 
\begin{equation}
	\label{eq:Aa} 
\begin{split}
Aa& \leq \sum_{j=0}^{\Js-1} \sumk \N^{-\frac{p}{2}}\norm{\psijkm}_{\Lp}^p\ind{\abs{\betacm}\geq \frac{1}{2}\thres}\\
&+\sum_{j=\Js}^{\Jn-1} \sumk \N^{-\frac{p}{2}}\ind{\abs{\betacm}\geq \frac{1}{2}\thres}\norm{\psijkm}_{\Lp}^p\\
&\eqqcolon Aa_1+Aa_2 . 
\end{split}
\end{equation}
Finally, in both the hard and soft thresholding frameworks, we have that
\begin{equation}
	\label{eq:Uu}
\begin{split}
Uu & = \sumjtrunc  \norm{\sumk \betacm\psijk}_{\Lp}^p\ind{\abs{\betacm}<2\thres}\\
&\leq\sum_{j=0}^{\Js-1} \sumk \abs{\betacm}^p B^{jd\left(\frac{p}{2}-1\right)}\ind{\abs{\betacm}< 2\thres} + \sum_{j=\Js}^{\Jn-1}\norm{\sumk \betacm\psijk}_{\Lp}^p\\
& \eqqcolon Uu_1+Uu_2. 
\end{split}
\end{equation}
\paragraph{\textbf{Regular zone.}} 
First of all, combining \eqref{eq:Au} and \eqref{eq:Ua} we choose $\gamma$ such that
$$
\gamma \geq \frac{2sp}{2\left(s+\mabs\right)+d}.
$$
and, then,
$$
\N^{-\frac{\gamma}{2}}\leq \N^{-\frac{sp}{2\left(s+\mabs\right)+d}}.
$$
Observe now \eqref{eq:Aa}. We have that
\begin{align*}
Aa_1 & = \sum_{j=0}^{\Js-1} \sumk n^{-\frac{p}{2}}\norm{\psijkm}_{\Lp}^p\\
 &\lesssim n^{-\frac{p}{2}} \sum_{j=0}^{\Js-1} \sumk B^{j\left(p\mabs+d\left(\frac{p}{2}-1 \right)\right)}\\
 &\lesssim n^{-\frac{p}{2}} \sum_{j=0}^{\Js-1} B^{\frac{jp}{2}\left(2\mabs+d\right)}\\
 &\lesssim n^{-\frac{p}{2}} B^{\frac{\Js p}{2}\left(2\mabs+d\right)}\\
 &\lesssim \left(\log n\right)^{\frac{p\left(2\mabs+d\right)}{2\left(2\left(s+\mabs\right)+d\right)}} n^{\frac{ps}{2\left(s+\mabs\right)+d}},
\end{align*}
since
\begin{align}
\notag\frac{p\left(2\mabs+d\right)}{2\left(2\left(s+\mabs\right)+d\right)}-\frac{p}{2}
& =\frac{p\left(2\mabs+d\right)-p\left(2\left(s+\mabs\right)+d\right)}{2\left(2\left(s+\mabs\right)+d\right)}\\
& = \frac{-sp}{2\left(s+\mabs\right)+d}\label{eq:diff}
\end{align}
On the other hand,
\begin{align*}
Aa_2 & =\sum_{j=\Js}^{\Jn-1} \sumk n^{-\frac{p}{2}}\norm{\psijkm}_{\Lp}^p\ind{\abs{\betacm}\geq \frac{1}{2}\thres}\\
& \lesssim n^{-\frac{p}{2}} \sum_{j=\Js}^{\Jn-1} \sumk B^{j\mabs p}\norm{\psijk}_{\Lp}^p \frac{\abs{\betacm}^p}{\thres^p}\\
&\lesssim \left(\log n\right) ^{-\frac{p}{2}}\sum_{j=\Js}^{\Jn-1} \sumk \abs{\betacm}^p \norm{\psijk}_{\Lp}^p\\
&\lesssim \left(\log n\right) ^{-\frac{p}{2}}\sum_{j=\Js}^{\Jn-1} B^{-jsp}\\
&\lesssim \left(\log n\right) ^{-\frac{p}{2}} B^{-\Js sp}.
\end{align*}
Then, it holds that 
\begin{align*}
	Aa\lesssim  \left(\log n\right)^{\frac{p\left(2\mabs+d\right)}{2\left(2\left(s+\mabs\right)+d\right)}} n^{-\frac{ps}{2\left(s+\mabs\right)+d}}
\end{align*}
Consider now \eqref{eq:Uu}. It holds that
\begin{align*}
Uu_1&=\sum_{j=0}^{\Js-1} \sumk \abs{\betacm}^p B^{jd\left(\frac{p}{2}-1\right)}\ind{\abs{\betacm}< 2\thres}\\
& \lesssim \sum_{j=0}^{\Js-1} 2^p \thres^p \sumk B^{jd\left(\frac{p}{2}-1\right)}\\
& \lesssim \sum_{j=0}^{\Js-1}\left(\frac{\log n}{n}\right)^{\frac{p}{2}} B^{\frac{jp}{2}\left(2\mabs+d\right)}\\
&\lesssim \left(\frac{\log n}{n}\right)^{\frac{p}{2}} B^{\frac{\Js p}{2}\left(2\mabs+d\right)}\\
&\lesssim \left(\log n\right)^{\frac{p}{2}} n^{-\frac{sp}{2\left(s+\mabs\right)+d}},
\end{align*}
in view of \eqref{eq:diff}, while
\begin{align*}
Uu_2 &= \sum_{j=\Js}^{\Jn-1}\norm{\sumk \betacm\psijk}_{\Lp}^p\\
&\lesssim \sum_{j=\Js}^{\Jn-1}B^{-jps}\\
&\lesssim B^{-\Js ps}
\end{align*}
Then, 
\begin{align*}
Uu\lesssim  \left(\log n\right)^{\frac{p}{2} }n^{-\frac{ps}{2\left(s+\mabs\right)+d}}.
\end{align*}
As far as $D$ is concerned, observe that 
\begin{align*}
	D^{\frac{1}{p}} &\leq \sum_{j\geq \Jn} \norm{\sumk\betacm\psijk}_{\Lp}\\
					& = \sum_{j\geq \Jn} \norm{\sumk\betac\psijkm}_{\Lp}\\
					&\leq \sum_{j\geq \Jn} B^{j\left(\mabs+d\left(\frac{1}{2}-\frac{1}{p}\right)\right)}\norm{\beta_{j,\cdot}}_{\ell^p}\\
					&\leq \sum_{j\geq \Jn}B^{-js}\norm{f}_{\besovm}\\
					&\leq B^{-\Jn s}\norm{f}_{\besovm}.					
\end{align*}
Thus $D\lesssim \left(\frac{\log n}{n}\right)^{\frac{sp}{d+2\mabs}}$
As in \cite{bkmpAoSb,dmg}, first of all, notice that for $p\leq r$, $$\besov \subseteq \besovgen{p}{q}{s},$$ so that we can always use $r=p$; consider then the case $p\geq r$, where the following embedding holds $$\besov \subseteq \besovgen{p}{q}{s-d\left(\frac{1}{r}-\frac{1}{p}\right)}.$$
Because in the regular zone $$r\geq \frac{\left(2\mabs+d\right)p}{2\left(s+\mabs\right)+d},$$ 
it follows that 
$$
\frac{s}{2\left(s+\mabs\right)+d}\leq \frac{sr}{\left(2\mabs+d\right)p}.$$
Thus, since $s>\frac{d}{r}$, it holds that 
\begin{align*}
\frac{s}{\left(2\mabs+d\right)}-\frac{d}{\left(2\mabs+d\right)}&\left(\frac{1}{r}-\frac{1}{p}\right)-\frac{s}{2\left(s+\mabs\right)+d}\\
&\geq
\frac{s}{\left(2\mabs+d\right)}-\frac{d}{\left(2\mabs+d\right)}\left(\frac{1}{r}-\frac{1}{p}\right)-\frac{sr}{\left(2\mabs+d\right)p}\\
&=
\frac{d}{\left(2\mabs+d\right)}\left(\frac{1}{r}-\frac{1}{p}\right)\left(\frac{sr}{d}-1\right)\\
&\geq 0,
\end{align*}
since $s>\frac{d}{r}$.

\paragraph{\textbf{Sparse zone.}} 
This proof is similar to the one above, so it is properly shortened for the sake of the brevity. As far as \eqref{eq:Au} and \eqref{eq:Ua} are concerned, in order to choose $\gamma$ such that
$$
Au+Ua\lesssim n^{-\frac{\gamma}{2}}\lesssim n^{-	\frac{p\left(s+d\left(\frac{1}{p}-\frac{1}{r}  \right)\right)}{2\left[\left(s+\mabs\right)+d\left(\frac{1}{2}-\frac{1}{r} \right) \right]}}.
$$
Observe now \eqref{eq:Aa}: note that 
\begin{align*}
Aa_2 & =0,
\end{align*}
since $$\ind{\abs{\betacm}\geq \frac{1}{2}\thres} \text{ for }j\geq \Js,$$
see also \cite{bkmpAoSb,durastanti6}.\\
Then, we have that
\begin{align*}
Aa_1 & = \sum_{j=0}^{\Js-1} \sumk n^{-\frac{p}{2}}\norm{\psijkm}_{\Lp}^p\ind{\abs{\betacm}\geq \frac{1}{2}\thres}\\
&\lesssim n^{-\frac{p}{2}} \sum_{j=0}^{\Js-1} \sumk \abs{\betacm}^r \thres^{-r} \norm{\psijk}_{\Lp}^pB^{j\mabs p} \\
&\lesssim \frac{n^{\frac{r-p}{2}}}{\left(\log n\right)^{\frac{r}{2}}} \sum_{j=0}^{\Js-1}B^{j\mabs \left(p-r\right)} B^{\frac{jd}{2}\left(p-r\right)} \sumk \abs{\betacm}^r  \norm{\psijk}_{\Lr}^r
 \\
 &\lesssim \frac{n^{\frac{r-p}{2}}}{\left(\log n\right)^{\frac{r}{2}}} \sum_{j=0}^{\Js-1}B^{j\mabs \left(p-r\right)} B^{\frac{jd}{2}\left(p-r\right)}B^{-jsr} \\
&\lesssim \frac{n^{\frac{r-p}{2}}}{\left(\log n\right)^{\frac{r}{2}}}  B^{\Js\left[\frac{p-r}{2}\left(2\mabs+d\right)-sr\right]}\\
&\lesssim \frac{n^{\frac{-p\left(s-d\left(\frac{1}{r}-\frac{1}{p}\right)\right)}{2\left(s+\mabs-d\left(\frac{1}{r}-\frac{1}{2}\right)\right)}}} {\left(\log n\right)^\delta},
\end{align*}
since
\begin{align}
\frac{\left(p-r\right)\left(2\mabs+d\right)-2sr}{4\left(s+\mabs+d\left(\frac{1}{2}-\frac{1}{r}\right)\right)}+\frac{r-p}{2}
& = \frac{-p\left(s-d\left(\frac{1}{r}-\frac{1}{p}\right)\right)}{2\left(s+\mabs-d\left(\frac{1}{r}-\frac{1}{2}\right)\right)}\label{eq:diff2},
\end{align}
and 
$$
\delta=\frac{\left(p-r\right)\left(2\mabs+d\right)-2sr}{4\left(s+\mabs+d\left(\frac{1}{2}-\frac{1}{r}\right)\right)}+\frac{r}{2}=\frac{\frac{p}{2}\left(2\mabs+ d\right)-d}{2\left(s+\mabs-d\left(\frac{1}{r}-\frac{1}{2}\right)\right)}.
$$
Consider now \eqref{eq:Uu}; on the one hand, it holds that
\begin{align*}
Uu_1&=\sum_{j=0}^{\Js-1} \sumk \abs{\betacm}^p B^{jd\left(\frac{p}{2}-1\right)}\ind{\abs{\betacm}< 2\thres}\\
& \lesssim \sum_{j=0}^{\Js-1}B^{jd\left(\frac{p}{2}-1\right)} \thres^{p-r} \sumk \abs{\betacm}^r\\
& \lesssim \left(\frac{n}{\log n}\right)^{\frac{r-p}{2}}\sum_{j=0}^{\Js-1}B^{\frac{j\left(p-r\right)}{2}\left(2\mabs + d\right)} \sumk \abs{\betacm}^r B^{{jd}\left(\frac{r}{2}-1\right)}\\
&\lesssim  \left(\frac{n}{\log n}\right)^{\frac{r-p}{2}}\sum_{j=0}^{\Js-1}B^{\frac{j\left(p-r\right)}{2}\left(2\mabs + d\right)} B^{-jsr}\\
&\lesssim \frac{n^{\frac{r-p}{2}}}{\left(\log n\right)^{\frac{r}{2}}}  B^{\Js\left[\frac{p-r}{2}\left(2\mabs+d\right)-sr\right]}\\
&\lesssim \frac{n^{\frac{-p\left(s-d\left(\frac{1}{r}-\frac{1}{p}\right)\right)}{2\left(s+\mabs-d\left(\frac{1}{r}-\frac{1}{2}\right)\right)}}} {\left(\log n\right)^\delta},
\end{align*}
in view of \eqref{eq:diff2}. On the other hand, analogously to in \cite{bkmpAoSb,dmg}, we define
$$
g=\frac{p\mabs +d\left(\frac{p}{2}-1\right)}{s+\mabs-d\left(\frac{1}{r}-\frac{1}{2}\right)}.
$$
In the sparse zone we have 
\begin{align*}
&g-r=\frac{\frac{1}{2}\left[p\left(2\mabs+d\right)-r\left(2s+2\mabs+d\right)\right]}{s+\mabs-d\left(\frac{1}{r}-\frac{1}{2}\right)}>0, 
\end{align*}
so that the embedding $\besovgen{r}{q}{s}\subseteq \besovgen{g}{q}{s-d\left(\frac{1}{r}-\frac{1}{g}\right)}$ holds. Furthermore, we easily obtain that
\begin{align*}
&p-g=\frac{p\left(s-d\left(\frac{1}{r}-\frac{1}{p}\right)\right)}{s+\mabs-d\left(\frac{1}{r}-\frac{1}{2}\right)}>0.
\end{align*}
Thus,
\begin{align*}
Uu_2&=\sum_{j=\Js}^{\Jn-1} B^{jd\left(\frac{p}{2}-1\right)} \sumk \abs{\betacm}^p \ind{\abs{\betacm}< 2\thres}\\
&\lesssim \left(\frac{\log n}{n}\right)^{\frac{p-g}{2}} \sum_{j=\Js}^{\Jn-1} B^{j\left(2\mabs+d\right)\left(\frac{p-g}{2}\right)} \sumk \abs{\betacm}^g B^{jd\left(\frac{g}{2}-1\right)}\\
&\lesssim \left(\frac{\log n}{n}\right)^{\frac{p-g}{2}}\sum_{j=\Js}^{\Jn-1} B^{j\left(2\mabs+d\right)\left(\frac{p-g}{2}\right)}  B^{-jg\left(s-d\left(\frac{1}{r}-\frac{1}{g}\right)\right)},\\
&\lesssim \Jn \left(\frac{\log n}{n}\right)^{\frac{p\left(s-d\left(\frac{1}{r}-\frac{1}{p}\right)\right)}{2\left(s+\mabs-d\left(\frac{1}{r}-\frac{1}{2}\right)\right)}},
\end{align*}
since
\begin{align*}
\left(2\mabs+d\right)\left(\frac{p-g}{2}\right)-g\left(s-d\left(\frac{1}{r}-\frac{1}{g}\right)\right)&=\left(\mabs+\frac{d}{2}\right)\left(p-g\right)-g\left(s-\frac{d}{r}\right)-d 
=0
\end{align*}
%
As far as $D$ is concerned, observe that 
\begin{align*}
D^{\frac{1}{p}} &\leq \sum_{j\geq \Jn} \norm{\sumk\betacm\psijk}_{\Lp}\\
& = \sum_{j\geq \Jn} \norm{\sumk\betac\psijkm}_{\Lp}\\
&\lesssim B^{-\Jn s-d\left(\frac{1}{r}-\frac{1}{p}\right)}\\
&\lesssim \left(\frac{n}{\log n}\right)^{-\frac{s-d\left(\frac{1}{r}-\frac{1}{p}\right)}{2\mabs +d}}.					
\end{align*}
Recalling that, in the sparse zone, $r\leq p$, it is straightforward to prove that 
$$
\frac{s-d\left(\frac{1}{r}-\frac{1}{p}\right)}{2\mabs +d}\geq \frac{s-d\left(\frac{1}{r}-\frac{1}{p}\right)}{2\left(s+\mabs-d\left(\frac{1}{r}-\frac{1}{2}\right)\right)},
$$
since, for $s >\frac{d}{r}$, 
$$
2\left(s+\mabs-d\left(\frac{1}{r}-\frac{1}{2}\right)\right)\geq d+2\mabs
$$
Thus $D\lesssim \left(\frac{\log n}{n}\right)^{\frac{sp}{d+2\mabs}}$, see also \cite{bkmpAoSb,dmg}.\\
Let us now consider the case $p= \infty$,  where we have
\begin{align*}
\Ex \left[\norm{\fest-\fm}_{\Linf}\right] 
& \lesssim \Ex\left[ \norm{ \sumjtrunc \left[\eta\left(\betamest, \thres\right)-\betac \right]\psijk  } _{\Linf}  \right]\\
&+\norm{ \sum_{j\geq \Jn}\sumk \betacm \psijk } _{\Linf}\\
&\eqqcolon\Sigma_\infty+D_\infty.
\end{align*}
Consider preliminarily $q=r=\infty$. Arguments similar to the one discussed above yield 
\begin{align*}
\Sigma_\infty & \lesssim \sum_{j=0}^{\Js-1} B^{\frac{j}{2}\left(2\mabs+d\right)} \Ex\left[\sup_{k=1,\ldots,\Kj}\abs{\widehat{\beta}_{j,k}-\betac}\right]+\sum_{j=\Js}^{\Jn-1} B^{\frac{j}{2}+d}\sup_{k=1,\ldots,\Kj}\abs{\betacm} + n^{-\frac{1}{2}}\\ 
& \lesssim \Js B^{\frac{\Js}{2}\left(2\mabs + d \right)}n^{-\frac{1}{2}}+ B^{-s\Js} + n^{-\frac{1}{2}}\\
& \lesssim n^{-\frac{s}{2\left(s+\mabs\right)+d}},
\end{align*}
As far as the deterministic error is concerned, it is easy to prove that 
\begin{align*}
D_\infty & \lesssim \sum_{j\geq \Jn} \norm{  \sumk\betacm \psijk } _{\Linf}   \\
&\lesssim B^{-s\Jn}\\
&\lesssim \left(\frac{n}{\log n}\right)^{-\frac{s}{2\left(s+\mabs\right)+d}}.
\end{align*}
Now, as in \cite{bkmpAoSb,dmg}, for arbitrary $q,r$, the result holds since $\besovgen{r}{q}{s}\subseteq \besovgen{\infty}{\infty}{s-\frac{d}{r}}$.
\end{proof}
	\noindent The idea of proof of Theorem \ref{thm:maintheoremlower} comes from \cite[Theorem 11]{bkmpAoSb}, see also \cite{LiuWang} and it is based on two crucial results, namely, the so-called Fano's lemma and Varshanov-Gilbert lemma (see, for example, \cite{tsybakov}). Given two probability measures $P$ and $Q$, defined on some probability space, their Kullback-Leibler divergence is given by
\begin{equation*}
KL\(P,Q\)=\begin{cases} \int \log \frac{\diff P}{\diff Q}\diff P = \int \frac{\diff P}{\diff Q} \log \frac{\diff P}{\diff Q}\diff Q & \text{if }P \ll Q\\
\infty & \text{otherwise} \end{cases},
\end{equation*}
Let $P$ and $Q$ be two probability measures on the $d$-torus with densities $f,g$ with respect to $\rho\(\diff \vartheta\)$. If $g$ is bounded below by some positive constant, it holds that 
\begin{equation*}
KL\(P,Q\) \lesssim \norm{f-g}_{L^2\(\Td\)}^2, 
\end{equation*}
see, for example, \cite[Equation (33)]{bkmpAoSb}. 
\begin{lemma}[Fano's lemma]\label{lemma:fano} For $t=1,\ldots,T$, let $\(\Omega, \mathcal{F},P_t \)$ be a set of probability spaces, and $A_k\in\mathcal{F}$. Let, furthermore,
	\begin{align*}
	\mathcal{L}_T=\inf_{t=1,\ldots,T} \frac{1}{T} \sum_{t^\prime \neq t} KL\(P_t,P_{t^\prime}\).
	\end{align*}		
	If, for $t \neq t^\prime$, $A_t \cap A_{t^\prime}=\emptyset$, then
	$$
	\sup_{t=1,\ldots,T} P_t \(A_t^c\) \geq \min\{\frac{1}{2},\sqrt{T} e^{-\frac{3}{e}}e^{-\mathcal{L}_T}\}.
	$$  
\end{lemma}
\begin{lemma}[Varshanov-Gilbert lemma]\label{lemma:vglemma}
	Let $\mathcal{E}=\{0,1\}^T$.
	Then, there exists a subset $\{\epsilon^{0},\ldots, \epsilon^{U}\}\subseteq\mathcal{E}$, where $\epsilon^{0}=\(0,\ldots,0\)$, such that $U\geq 2^{T/8}$ and 
	$$
	\sum_{t=1}^T \abs{\epsilon_t^{u}-\epsilon_t^{u^\prime}} \geq \frac{T}{8}, \quad 0\leq u\neq u^\prime \leq U.
	$$
\end{lemma}
\begin{proof}[Proof of Theorem \ref{thm:maintheoremlower}] As aforementioned, this proof follows strictly the approach developed by \cite[Theorem 11]{bkmpAoSb}.\\
	Let us fix $j \geq 0$ and consider the family $\mathcal{A}_j$ of densities taking the form
	$$
		f_\epsilon = \frac{1}{\(2\pi\)^d} + \zeta \sum_{k\in A_j}\epsilon_{j,k}\psijk,  
	$$
	where $A_j \subseteq \{1,\ldots,\Kj\}$ is chosen so that \eqref{eq:inf} in Lemma \ref{lemma:supinf} holds, $\epsilon_k \in \{0,1\}$, and $\zeta>0$ must ensure that all the densities in $\mathcal{A}_j$ are positive. It is sufficient that $\gamma\lesssim B^{-j\frac{d}{2}}$, since 
	\begin{align*}
		\abs{f_\epsilon}\geq & \frac{1}{\(2\pi\)^d} - \abs{\gamma}\abs{\sum_{k\in A_j}\epsilon_{j,k}\psijk}\\
		\geq & \frac{1}{\(2\pi\)^d}  - \abs{\gamma}c_\infty B^{j\frac{d}{2}},
	\end{align*}
	see again \cite{bkmpAoSb}.
	Using now \eqref{eq:sumbeta2} in Lemma \ref{lemma:sumbeta} yields 
	\begin{align*}
		f_\epsilon^{\(m\)}\(\vartheta\)= & \zeta \sum_{k\in A_j}\epsilon_{j,k}\psijkm\(\vartheta\), \quad \vartheta \in \Td.
	\end{align*}
	To ensure that $f_\epsilon\in \besovgen{r}{q}{s+\mabs}\(G\)$, we impose that 
	$$
		\abs{\zeta} \leq G  B^{-j\(s+\mabs+\frac{d}{2}\)}.
	$$
	Indeed, since 
	$$
	\(\sumk\abs{\epsilon_{j,k}}^r\)^{\frac{1}{r}} \leq \(\sumk 1\)^{\frac{1}{r}} \lesssim B^{\frac{jd}{r}}
	$$
	\begin{align*}
		\norm{f_\epsilon^{\(\mabs\)}}_{\besov} & = \abs{\zeta} B^{j\(s+\mabs+d\(\frac{1}{2}-\frac{1}{r}\)\)}\(\sumk\abs{\epsilon_{j,k}}^r\)^{\frac{1}{r}}\\
		&\lesssim \abs{\zeta} B^{j\(s+\mabs+\frac{d}{2}\)}.
	\end{align*} 
	Now, for $f_\epsilon,f_{\epsilon^\prime} \in \mathcal{A_j}$, 
	\begin{align*}\notag
		\norm{f^{\(m\)}_\epsilon - f^{\(m\)}_{\epsilon^\prime}}_{\Ltwo}^2 & \lesssim \zeta^2 \sum_{k \in A_j}\abs{\epsilon_{j,k}-\epsilon_{j,k}^{\prime}}^2B^{2j\mabs}\\
		& \lesssim B^{-2js}.
	\end{align*}	
	Using \eqref{eq:inf} in Lemma \ref{lemma:supinf} yields
	\begin{equation}\label{eq:lpuno}
			\norm{f^{\(m\)}_\epsilon - f^{\(m\)}_{\epsilon^\prime}}_{\Lp}\geq \abs{\zeta}\(\sum_{k \in A_j}\abs{\epsilon_{j,k}-\epsilon_{j,k}^{\prime}}^p\norm{\psijkm}_{\Lp}^p\)^{\frac{1}{p}}.
	\end{equation}
	According to Lemma \ref{lemma:vglemma}, there exists a finite subset of $\mathcal{A}_j$, whose elements are given by
	\begin{equation*}
			f_{\epsilon^u} = \frac{1}{\(2\pi\)^d} + \zeta \sum_{k\in A_j}\epsilon_{j,k}^u\psijk,  
	\end{equation*}
	where $\{\epsilon_{j,\cdot}^u:u=1,\ldots,U\}$ is such that $U\geq 2^{cB^{dj}}$ and 
	$$
		\sum_{k \in A_j} \abs{\epsilon_{j,k}^u-\epsilon_{j,k}^u} \gtrsim B^{jd}.
	$$
	Thus, using \eqref{eq:lpuno} yields
		\begin{equation*}
	\norm{f^{\(m\)}_\epsilon - f^{\(m\)}_{\epsilon^\prime}}_{\Lp}\gtrsim \abs{\zeta} B^{j\mabs+d\(\frac{1}{2}-\frac{1}{p}\)} B^{\frac{jd}{p}} \approx B^{-js},
	\end{equation*} 
	which implies that the events $$A_{\epsilon^u}=\{\norm{\fmest-f^{\(m\)}_{\epsilon^{u}}}<\frac{1}{2}B^{-js}\}, \quad u=1,\ldots,U,$$ 
	are disjoint.\\
	Fixed a density function $f$, consider now the probability measure $P^\N_f$, corresponding to the density
	$$
		f^{n}\(x\)=f\(x_1\)\cdot\ldots \cdot f\(x_\N\).
	$$ 
	Using Lemma \ref{lemma:fano} leads to
	$$
	\sup_{u=1,\ldots,U} P^{\N}_{f_{\epsilon^u}}\(A_{\epsilon^u}^c\)\geq \min \{\frac{1}{2},\sqrt{U} e^{-\frac{3}{e}}e^{-\mathcal{L}_U}\},
	$$
	such that 
	$$
		\sup_{u=1,\ldots,U}\Ex\[\norm{\fmest,f_{\epsilon^u}^{\(m\)}}_{\Lp}^p\]\geq \frac{B^{-jps}}{2}\sup_{u=1,\ldots,U}P^\N_{f_{\epsilon^u}}\(A_{\epsilon^u}^c\) \gtrsim B^{-jsp} \min \{\frac{1}{2},\sqrt{U} e^{-\frac{3}{e}}e^{-\mathcal{L}_U}\}. 
	$$
	Observing that
	$$
	KL\(P^\N_1,P^\N_2\)=\sum_{i=1}^{N}\int f_1\(x_i\)\log \frac{f_1\(x_i\)}{f_2\(x_i\)}\diff x_i = \N KL\(f_1,f_2\),
	$$
	it holds that 
	$$
		\mathcal{L}_U\lesssim \inf_{u=1,\ldots,U} \frac{\N}{U}\sum_{u^\prime\neq u} KL\(f_{\epsilon^u},f_{\epsilon^{u^\prime}}\) \lesssim \frac{\N}{U}\sum_{u=1}^{U} KL\(f_{\epsilon^u},f_{\epsilon^0}\),
	$$
	where
	\begin{align*}
	KL\(f_{\epsilon^u},f_{\epsilon^0}\) & = \int \frac{1}{f_{\epsilon^0}\(\vartheta\)} \abs{f_{\epsilon^u}\(\vartheta\)-f_{\epsilon^0}\(\vartheta\)}^2 \rho\(\diff \vartheta\)\\
	&= \(2\pi\)^d \abs{\zeta}^2 \norm{\sum_{k\in A_j} \epsilon^{u}_{j,k}\psijk}_{\Ltwo}^{2}\\
	&\leq G\(2\pi\)^d B^{-2j\(s+\mabs+\frac{d}{2}\)}B^{dj}\\
	&\lesssim B^{-2j\(s+\mabs\)},
	\end{align*}
	and, then, 
	$$
	\mathcal{L}_U\lesssim \N B^{-2j\(s+\mabs\)}
	$$
	while
	$$
		\sqrt{U} \approx 2^{cB^{jd}},
	$$
	Then, we choose $j$ such that $\N B^{-2js} \approx B^{jd}$, that is, $j\in \naturals$ such that $B^{j}\approx \N^{\frac{1}{2\(s+\mabs\)+d}}$.
	Hence, it holds that
	$$
	\sup_{u=1,\ldots,U} \Ex\[\norm{\fmest-f^{\(m\)_{\epsilon^{u}}}}_{\Lp}^p\]\gtrsim B^{-jsp}\approx \N^{-\frac{sp}{2\(s+\mabs\)+d}}. 
	$$
	Consider now two densities 
	$$
	f_0=\frac{1}{\(2\pi\)^d}+\zeta \psijk; \quad f_1= \frac{1}{\(2\pi\)^d}+\zeta \psi_{j,k^\prime}, 
	$$
	where $\abs{\zeta} \lesssim B^{-\frac{jd}{2}}$ ensures that the two densities are positive. If, additionally, $\abs{\zeta} \leq GB^{-j\(s+\mabs+d\(\frac{1}{2}-\frac{1}{r}\)\)},$ then both $f_0$ and $f_1$ belong to the Besov ball $\besovgen{r}{q}{s+\mabs}\(G\)$. Again, $$KL\(f_0,f_1\)\approx \zeta^2,$$
	while, if $P_0$ e $P_1$ denote the probability measures whose densities are given by the $\N$-products of $f_0$ and $f_1$ respectively, it holds that
	$$
	KL\(P_0,P_1\)\approx \N\zeta^2.
	$$    
	using again Lemma \ref{lemma:supinf}, we have that
	\begin{align*}
		\norm{f_0^{\(m\)}-f_1^{\(m\)}}_{\Lp} &= \abs{\zeta} \norm{\psijkm-\psi^{\(m\)}_{j,k^\prime}}_{\Lp}\\
		& \gtrsim \zeta B^{j\mabs}B^{jd\(\frac{1}{2}-\frac{1}{p}\)}\\
		& \gtrsim B^{-j\(s+d\(\frac{1}{p}-\frac{1}{r}\)\)}.
	\end{align*}
	Thus, the events $\{\norm{\fmest-f_i^{\(m\)}}_\Lp\geq B^{-j\(s+d\(\frac{1}{p}-\frac{1}{r}\)\)}\}$ are disjoint.
	As in \cite{bkmpAoSb}, choosing $\zeta=\N^{-\frac{1}{2}}$, so that $KL\(f_0,f_1\)\approx \N$, implies that $j \approx \frac{1}{2\(s+\mabs + d\(\frac{1}{2}-\frac{1}{r}\)\)}.$ Hence, using Fano's lemma yields
	\begin{align*}
	 \sup_{i=0,1} \Ex\[\norm{\fmest-f_i^{\(m\)}}_\Lp\] & \gtrsim B^{-j\(s+d\(\frac{1}{p}-\frac{1}{r}\)\)}\\
	 & \approx \N^{-\frac{\(s+d\(\frac{1}{p}-\frac{1}{r}\)\)}{2\(s+\mabs+d\(\frac{1}{2}-\frac{1}{r}\)\)}}.	  
	\end{align*}
	Finally, combining these results and checking for which sets of the Besov parameters one rate is
	larger than the other one completes the proof of the theorem. 
\end{proof}
\section*{Acknowledgments}
The authors would like to thank the Editor and the anonymous referees
for their comments and constructive suggestions. Claudio Durastanti acknowledges the research project RM120172B7A31FFA ``Costruzione di basi multiscala e trasformate wavelet per applicazioni in ambito numerico e statistico'', funded by Sapienza Università di Roma.
Nicola Turchi was supported by FNR OPEN FoRGES (R-AGR-3376-10-C) at Université du Luxembourg.
\section*{Disclosure statement}
	No potential conflict of interest was reported by the authors.

\begin{thebibliography}{}

\bibitem[Baldi et~al., 2009a]{bkmpAoSb}
Baldi, P., Kerkyacharian, G., Marinucci, D., and Picard, D. (2009a).
\newblock Adaptive density estimation for directional data using needlets.
\newblock {\em Ann. Statist.}, 37(6A):3362--3395.

\bibitem[Baldi et~al., 2009b]{bkmpAoS}
Baldi, P., Kerkyacharian, G., Marinucci, D., and Picard, D. (2009b).
\newblock Asymptotics for spherical needlets.
\newblock {\em Ann. Statist.}, 37:1150--1171.

\bibitem[Baldi et~al., 2009c]{bkmpBer}
Baldi, P., Kerkyacharian, G., Marinucci, D., and Picard, D. (2009c).
\newblock Subsampling needlet coefficients on the sphere.
\newblock {\em Bernoulli}, 15:438--463.

\bibitem[Bott et~al., 2017]{bott}
Bott, A.-K., Felber, T., Kohler, M., and Kirstl, L. (2017).
\newblock Estimation of time--dependent density.
\newblock {\em J. Statist. Plann. Inference}, 180:81--107.

\bibitem[Bourguin and Durastanti, 2017]{bd2}
Bourguin, S. and Durastanti, C. (2017).
\newblock On high-frequency limits of {$U$}-statistics in besov spaces over
  compact manifolds.
\newblock {\em Illinois J. Math.}, 61(1-2):97--125.

\bibitem[Bourguin and Durastanti, 2018]{bd1}
Bourguin, S. and Durastanti, C. (2018).
\newblock On normal approximations for the two-sample problem on
  multidimensional tori.
\newblock {\em Journal of Statistical Planning and Inference}, 196:56 -- 69.

\bibitem[Bourguin et~al., 2016]{bdmp}
Bourguin, S., Durastanti, C., Marinucci, D., and Peccati, G. (2016).
\newblock Gaussian approximations of nonlinear statistics on the sphere.
\newblock {\em J. Math. Anal. Appl.}, 436:1121--1148.

\bibitem[Cammarota and Marinucci, 2015]{cammar}
Cammarota, V. and Marinucci, D. (2015).
\newblock On the limiting behaviour of needlets polyspectra.
\newblock {\em Ann. Inst. H. Poincar\'e Probab. Statist.}, 51:1159--1189.

\bibitem[Chacon et~al., 2012]{chacon1}
Chacon, J., Duong, and T.~Wand, M. (2012).
\newblock Asymptotics for general multivariate kernel density derivative
  estimators.
\newblock {\em Stat. Sin.}, 21:807--840.

\bibitem[Cheng, 1995]{cheng}
Cheng, Y. (1995).
\newblock Mean shift, mode seeking, and clustering.
\newblock {\em IEEE Trans. Pattern Anal.}, 17:790--799.

\bibitem[Comaniciu and Meer, 2002]{comanciu}
Comaniciu, D. and Meer, P. (2002).
\newblock Mean shift: A robust approach toward feature space analysis.
\newblock {\em IEEE Trans. Pattern Anal.}, 24:603--619.

\bibitem[Di~Marzio et~al., 2011]{dimarzio}
Di~Marzio, M., Panzera, A., and Taylor, C. (2011).
\newblock Kernel density estimation on the torus.
\newblock {\em J. Statist. Plann. Inference}, 141:2156--2173.

\bibitem[Donoho et~al., 1996]{donoho1}
Donoho, D., Johnstone, I., Kerkyacharian, G., and D., P. (1996).
\newblock Density estimation by wavelet thresholding.
\newblock {\em Ann. Statist.}, 24:508--539.

\bibitem[Durastanti, 2013]{durastanti2}
Durastanti, C. (2013).
\newblock Block thresholding on the sphere.
\newblock {\em Sankhya A}, 77:153--185.

\bibitem[Durastanti, 2016]{durastanti6}
Durastanti, C. (2016).
\newblock Adaptive global thresholding on the sphere.
\newblock {\em J. Multivariate Anal.}, 151:110--132.

\bibitem[Durastanti, 2017]{durastanti1}
Durastanti, C. (2017).
\newblock Tail behavior of mexican needlets.
\newblock {\em J. Math. Anal. Appl.}, 447:716--735.

\bibitem[Durastanti et~al., 2011]{dmg}
Durastanti, C., Geller, D., and Marinucci, D. (2011).
\newblock Adaptive nonparametric regression of spin fiber bundles on the
  sphere.
\newblock {\em J. Multivariate Anal.}, 104:16--38.

\bibitem[Durastanti et~al., 2013]{dlmejs}
Durastanti, C., Lan, X., and Marinucci, D. (2013).
\newblock Needlet-whittle estimates on the unit sphere.
\newblock {\em Electron. J. Stat.}, 7:597--646.

\bibitem[Durastanti et~al., 2014]{dmp}
Durastanti, C., Marinucci, D., and Peccati, G. (2014).
\newblock Normal approximations for wavelet coefficients on spherical poisson
  fields.
\newblock {\em J. Math. Anal. Appl.}, 409:212--227.

\bibitem[Durastanti et~al., 2021]{dmt21}
Durastanti, C., Marinucci, D., and Todino, A. (2021).
\newblock Flexible--bandwidth needlets.
\newblock preprint.

\bibitem[Efromovich, 1985]{efrom}
Efromovich, S.~Y. (1985).
\newblock Nonparametric estimation of a density of unknown smoothness.
\newblock {\em Theory Probab. Appl.}, 30:557--661.

\bibitem[{Ferdosi, B. J.} et~al., 2011]{ferdosi}
{Ferdosi, B. J.}, {Buddelmeijer, H.}, {Trager, S. C.}, {Wilkinson, M. H. F.},
  and {Roerdink, J. B. T. M.} (2011).
\newblock Comparison of density estimation methods for astronomical datasets.
\newblock {\em A\&A}, 531:A114.

\bibitem[Fukunaga and Hostetler, 1975]{fukunaga}
Fukunaga, K. and Hostetler, L. (1975).
\newblock The estimation of the gradient of a density function, with
  applications in pattern recognition.
\newblock {\em IEEE T. Inform. Theory}, 21:32--40.

\bibitem[Garcia-Portugues et~al., 2013]{gp13}
Garcia-Portugues, E., Crujeiras, R., and Gonzalez-Manteiga, W. (2013).
\newblock Exploring wind direction and {SO$_2$} concentration by
  circular-linear density estimation.
\newblock {\em Stoch. Environ. Res. Risk Assess.}, 27:1055--1067.

\bibitem[Gautier and Le~Pennec, 2013]{gautier}
Gautier, R. and Le~Pennec, E. (2013).
\newblock Adaptive estimation in the nonparametric random coefficients binary
  choice model by needlet thresholding.
\newblock Submitted, arXiv:1106.3503.

\bibitem[Geller and Marinucci, 2010]{gelmar}
Geller, D. and Marinucci, D. (2010).
\newblock Spin wavelets on the sphere.
\newblock {\em J. Fourier Anal. Appl.}, 16:840--884.

\bibitem[Geller and Mayeli, 2009]{gm2}
Geller, D. and Mayeli, A. (2009).
\newblock Nearly tight frames and space-frequency analysis on compact
  manifolds.
\newblock {\em Math. Z.}, 263:235--264.

\bibitem[Geller and Pesenson, 2011]{gelpes}
Geller, D. and Pesenson, I. (2011).
\newblock Band-limited localized parseval frames and {B}esov spaces on compact
  homogeneous manifolds.
\newblock {\em J. Geom. Anal.}, 21:334--371.

\bibitem[Genovese et~al., 2016]{GPVW}
Genovese, C., Perone-Pacifico, M., Verdinelli, I., and Wasserman, L. (2016).
\newblock Nonparametric inference for density modes.
\newblock {\em J. R. Stat. Soc. B}, 78:99--126.

\bibitem[Grafakos, 2008]{grafokos}
Grafakos, L. (2008).
\newblock {\em Classical Fourier analysis}.
\newblock Springer.

\bibitem[Grith et~al., 2018]{economy}
Grith, M., Wagner, H., Hardle, W.~K., and Kneip, A. (2018).
\newblock Functional principal component analysis for derivatives of
  multivariate curves.
\newblock {\em Statist. Sinica}, 28(4):2469--2496.

\bibitem[Hardle et~al., 1997]{WASA}
Hardle, W., Kerkyacharian, G., Picard, D., and Tsybakov, A. (1997).
\newblock {\em Wavelets, approximations and statistical applications}.
\newblock Springer.

\bibitem[Hosseinioun et~al., 2011]{HDN11}
Hosseinioun, N., Doosti, H., and Nirumand, H.~A. (2011).
\newblock Nonparametric estimation of a multivariate probability density for
  mixing sequences by the method of wavelets.
\newblock {\em Ital. J. Pure Appl. Math.}, 28:31--40.

\bibitem[Hosseinioun et~al., 2012]{HDN12}
Hosseinioun, N., Doosti, H., and Nirumand, H.~A. (2012).
\newblock Nonparametric estimation of the derivatives of a density by the
  method of wavelet for mixing sequences.
\newblock {\em Statistical Papers}, 53:195--203.

\bibitem[Kerkyacharian et~al., 2012]{knp}
Kerkyacharian, G., Nickl, R., and Picard, D. (2012).
\newblock Concentration inequalities and confidence bands for needlet density
  estimators on compact homogeneous manifolds.
\newblock {\em Probab. Theory Relat. Fields}, 153:363--404.

\bibitem[Kerkyacharian and Picard, 1992]{Kerkypicard92}
Kerkyacharian, G. and Picard, D. (1992).
\newblock Density estimation in {B}esov spaces.
\newblock {\em Statist. Probab. Lett.}, 13:15--24.

\bibitem[Kerkyacharian and Picard, 1993]{Kerkypicard93}
Kerkyacharian, G. and Picard, D. (1993).
\newblock Density estimation by kernel and wavelet methods: optimality of
  {B}esov spaces.
\newblock {\em Statist. Probab. Lett.}, 18:327--326.

\bibitem[Kerkyacharian and Picard, 2004]{Kerkypicard}
Kerkyacharian, G. and Picard, D. (2004).
\newblock Regression in random design and warped wavelets.
\newblock {\em Bernoulli}, 10:1053--1105.

\bibitem[Kerkyacharian et~al., 1996]{kpt96}
Kerkyacharian, G., Picard, D., and Tribouley, K. (1996).
\newblock {$L_p$} adaptive density estimation.
\newblock {\em Bernoulli}, 2:229--247.

\bibitem[Liu and Wang, 2013]{LiuWang}
Liu, Y. and Wang, H. (2013).
\newblock Wavelet estimations for density derivatives.
\newblock {\em Sci. China Math.}, 56:483--495.

\bibitem[Marinucci and Peccati, 2011]{MaPeCUP}
Marinucci, D. and Peccati, G. (2011).
\newblock {\em Random Fields on the Sphere: Representations, Limit Theorems and
  Cosmological Applications}.
\newblock Cambridge University Press.

\bibitem[Narcowich et~al., 2006a]{npw2}
Narcowich, F.~J., Petrushev, P., and Ward, J.~D. (2006a).
\newblock Decomposition of {Besov} and {Triebel-Lizorkin} spaces on the sphere.
\newblock {\em J. Funct. Anal.}, 238:530--564.

\bibitem[Narcowich et~al., 2006b]{npw1}
Narcowich, F.~J., Petrushev, P., and Ward, J.~D. (2006b).
\newblock Localized tight frames on spheres.
\newblock {\em SIAM J. Math. Anal.}, 38:574--594.

\bibitem[Parzen, 1962]{parzen}
Parzen, E. (1962).
\newblock On estimation of a probability density function and mode.
\newblock {\em Ann. Math. Statist.}, 33:1065--1076.

\bibitem[Pesenson, 2013]{pesenson}
Pesenson, I. (2013).
\newblock Multiresolution analysis on compact riemannian manifolds.
\newblock In {\em Multiscale analysis and nonlinear dynamics, Rev. Nonlinear
  Dyn. Complex.}, pages 65--82. Wiley-VCH, Weinheim.

\bibitem[Prakasa~Rao, 1990]{prakasaISE}
Prakasa~Rao, B. (1990).
\newblock Estimation of the integrated squared density derivatives by wavelets.
\newblock {\em Bull. Inform. Cybernet.}, 31(1):47--65.

\bibitem[Prakasa~Rao, 1996]{prakasa0}
Prakasa~Rao, B. (1996).
\newblock Nonparametric estimation of the derivatives of a density by the
  method of wavelets.
\newblock {\em Bull. Inform. Cybernet.}, 28(1):91--100.

\bibitem[Prakasa~Rao, 2000]{prakasa1}
Prakasa~Rao, B. (2000).
\newblock Nonparametric estimation of partial derivatives of a multivariate
  probability density by the method of wavelets.
\newblock In {\em Asymptotics in Statistics and Probability}.

\bibitem[Prakasa~Rao, 2017]{prakasa17}
Prakasa~Rao, B. (2017).
\newblock Wavelet estimation for derivative of a density in a garch-type model.
\newblock {\em Comm. Statist. Theory Methods}, 46(5):2396--2410.

\bibitem[Prakasa~Rao, 2018]{Prakasa18}
Prakasa~Rao, B. (2018).
\newblock Wavelet estimation for derivative of a density in the presence of
  additive noise.
\newblock {\em Braz. J. Probab. Stat.}, 32:834--850.

\bibitem[Schuster, 1969]{schuster}
Schuster, E. (1969).
\newblock Estimation of a probability density function and its derivatives.
\newblock {\em Ann. Math. Statist.}, 40:1187--1195.

\bibitem[Sengupta and Ugwuowo, 2011]{biology}
Sengupta, A. and Ugwuowo, F. (2011).
\newblock A classification method for directional data with application to the
  human skull.
\newblock {\em Comm. Statist. Theory Methods}, 40(3):457--466.

\bibitem[Shevchenko and Todino, 2021]{st22}
Shevchenko, R. and Todino, A. (2021).
\newblock Asymptotic behaviour of level sets of needlet random fields.
\newblock preprint.

\bibitem[Silverman, 1986]{silverman}
Silverman, B. (1986).
\newblock {\em Density estimation for statistics and data analysis}.
\newblock Chapman and Hall.

\bibitem[Singh, 1977]{singh}
Singh, R. (1977).
\newblock Applications of estimators of a density and its derivatives to
  certain statistical problems.
\newblock {\em J. R. Stat. Soc. B}, 39:357,363.

\bibitem[Tribouley, 1995]{trib95}
Tribouley, K. (1995).
\newblock Practical estimation of multivariate densities using wavelet methods.
\newblock {\em Stat. Neerl}, 49(1):41--62.

\bibitem[Tsybakov, 2009]{tsybakov}
Tsybakov, A. (2009).
\newblock {\em Introduction to nonparametric estimation}.
\newblock Springer.

\bibitem[Xu, 2020]{Xu}
Xu, J. (2020).
\newblock Wavelet thresholding estimation of density derivatives from a
  negatively associated size-biased sample.
\newblock {\em Int. J. Wavelets Multiresolut. Inf. Process.}, page 2050016.

\bibitem[Zhang et~al., 2020]{energies}
Zhang, L., Xie, L., Han, Q., Wang, Z., and Huang, C. (2020).
\newblock Probability density forecasting of wind speed based on quantile
  regression and kernel density estimation.
\newblock {\em Energies}, 13(22).

\end{thebibliography}

\end{document}